\documentclass[12pt]{amsart}
\usepackage{amssymb,amsmath,mathrsfs}
\usepackage[colorlinks=true,urlcolor=blue,
citecolor=red,linkcolor=blue,linktocpage,pdfpagelabels,
bookmarksnumbered,bookmarksopen]{hyperref}
\usepackage[active]{srcltx}
\usepackage{verbatim}
\usepackage{epsfig,graphicx,color,mathrsfs}
\usepackage{graphicx}
\usepackage{amsmath,amssymb,amsthm,amsfonts}
\usepackage{amssymb}
\usepackage[english]{babel}
\usepackage[left=2.7cm,right=2.7cm,top=3cm,bottom=3cm]{geometry}
\usepackage{cite}
\usepackage[colorlinks=true,urlcolor=blue,
citecolor=red,linkcolor=blue,linktocpage,pdfpagelabels,
bookmarksnumbered,bookmarksopen]{hyperref}

\newcommand{\N}{{\mathbb N}}

\newcommand{\R}{{\mathbb R}}

\def \LL {\mathcal{L}}
\def \de {\partial}

\usepackage{amsthm}
\theoremstyle{definition}
\newtheorem{defn}{Definition}[section]

\newtheorem{rem}[defn]{Remark}

\theoremstyle{plain}
\newtheorem{thm}[defn]{Theorem}
\newtheorem{prop}[defn]{Proposition}
\newtheorem{lem}[defn]{Lemma}

\numberwithin{equation}{section}

\date{\today}
\makeatletter
\@namedef{subjclassname@2020}{\textup{2020} Mathematics Subject Classification}
\makeatother

\begin{document}
\title[On mixed local-nonlocal problems with Hardy potential]{On mixed local-nonlocal problems with Hardy potential}

\author[S.\,Biagi]{Stefano Biagi}
\author[F.\,Esposito]{Francesco Esposito}
\author[L.\,Montoro]{Luigi Montoro}
\author[E.\,Vecchi]{Eugenio Vecchi}

\address[S.\,Biagi]{Politecnico di Milano - Dipartimento di Matematica
\newline\indent
Via Bonardi 9, 20133 Milano, Italy}
\email{stefano.biagi@polimi.it}

\address[F.\,Esposito]{Dipartimento di Matematica e Informatica, Universit\`a della Calabria
\newline\indent
Ponte Pietro Bucci 31B, 87036 Arcavacata di Rende, Cosenza, Italy}
\email{francesco.esposito@unical.it}

\address[L.\,Montoro]{Dipartimento di Matematica e Informatica, Universit\`a della Calabria
\newline\indent
Ponte Pietro Bucci 31B, 87036 Arcavacata di Rende, Cosenza, Italy}
\email{montoro@mat.unical.it}

\address[E.\,Vecchi]{Dipartimento di Matematica, Universit\`a di Bologna
\newline\indent
Piazza di Porta San Donato 5, 40126 Bologna, Italy}
\email{eugenio.vecchi2@unibo.it}

\thanks{The authors are members of INdAM. 
S.Biagi and E.Vecchi are partially supported by 
the PRIN 2022 project 2022R537CS \emph{$NO^3$ - Nodal Optimization, Nonlinear elliptic equations, Nonlocal geometric problems, with a focus on regularity}, founded by the European Union - Next Generation EU. 
F. Esposito 
and L. Montoro are partially supported by the PRIN 2022 project P2022YFAJH 
{\em Linear and Nonlinear PDE's: New directions and Applications}.
  F. Esposito and E. Vecchi are partially supported by the INdAM-GNAMPA Project 
  \emph{Esistenza, unicità e regolarità di soluzioni per problemi singolari}.}




\begin{abstract}
In this paper we study the effect of the Hardy potential on existence, uniqueness and optimal summability of solutions of the mixed local-nonlocal elliptic problem
$$-\Delta u + (-\Delta)^s u - \gamma \frac{u}{|x|^2}=f \text{ in } \Omega, \ u=0 \text{ in } \R^n \setminus \Omega,$$
where $\Omega$ is a bounded domain in $\R^n$ containing the origin and $\gamma> 0$. In particular, we will discuss the existence, non-existence and uniqueness of solutions in terms of the summability of $f$ and  of the value of the parameter $\gamma$.

\end{abstract}

\maketitle

\medskip

\section{Introduction}\label{intro}
Elliptic and parabolic PDEs with the Hardy potential $|x|^{-2}$, where $x\in \mathbb{R}^n$ with $n\geq 3$, are nowadays a classical topic of investigation. We refer to the monograph \cite{PeralBook} for a detailed description of the mathematical and physical reasons that led to the study of such PDEs as well as most of the known result both in the model local and nonlocal case, namely when the leading operator is either $-\Delta$ or $(-\Delta)^s$ with $s\in (0,1)$. In the latter case, the interesting Hardy potential is given by $|x|^{-2s}$. 
Our aim here is to start the study of problems of the form 
\begin{equation*}
\left\{\begin{array}{rl}
\mathcal{L}u -\gamma\, \dfrac{u}{|x|^{2}} = f(x) & \textrm{in } \Omega,\\
u=0 & \textrm{in } \mathbb{R}^{n}\setminus \Omega,
\end{array}\right.
\end{equation*}
\noindent where $\Omega \subset \mathbb{R}^{n}$ is an open and bounded set with smooth enough boundary $\partial \Omega$, $0\in \Omega$, $f$ belongs to some suitable $L^{m}(\Omega)$ and $\gamma>0$. It will be soon clarified that the choice of the {\it local} Hardy potential is the most proper in this case, and this will immediately provide a natural upper bound for $\gamma$.
Finally, we introduce the operator
$$\mathcal{L} := -\Delta + (-\Delta)^{s},$$
where $(-\Delta)^s$ denotes the fractional Laplacian of order $s\in (0,1)$, i.e.,
$$(-\Delta)^s u = C_{n,s}\,\lim_{\varepsilon\to 0}\int_{\{|x-y|\geq\varepsilon\}}
\frac{u(x)-u(y)}{|x-y|^{n+2s}}\,dy,$$
and $C_{n,s} > 0$ is a suitable normalization constant (see, precisely, Remark \ref{rem:X12prop}). 

The above mixed local-nonlocal o\-pe\-ra\-tor is a particular instance of a general class of 
o\-pe\-ra\-tors firstly studied in the 60's \cite{BCP, Cancelier} in connection with the validity of maximum principle and 
more recently from a probabilistic point of view, see, e.g., \cite{CKSV} and the references therein.
Recently there has been a renewed interest in problems presenting both a local and nonlocal nature, mainly due to 
their interest in the applications see, e.g., \cite{DPLV2}. The most common thread of the more recent contributions 
concerns the development of a regularity theory making mainly use of purely analytical techniques, 
see, e.g., \cite{BDVV, DeFMin, GarainKinnunen, GarainLindgren, SVWZ, SVWZ2} and the references therein, so to be 
able to treat PDEs where the leading operator is $\mathcal{L}$.\\
As already announced, our interest in this paper is to start the study of mixed local-nonlocal PDEs with a singular potential. We just mention that other singular PDEs driven by $\mathcal{L}$ have been considered so far, see e.g. \cite{ArRa, Garain, BV2}. Coming back to our issues and in order to clarify the choice of the purely local Hardy potential, the starting point is the Hardy inequality, originally proved in \cite{Hardy_Littlewood} in the one 
dimensional case and subsequently extended in various directions:
\begin{equation}\label{eq:NonOptimalHardy}
C \, \int_{\mathbb{R}^{n}}\dfrac{u^2}{|x|^2}\, dx \leq  \int_{\mathbb{R}^{n}}|\nabla u|^2\, dx, \quad u\in 
C^{\infty}_{0}(\mathbb{R}^{n}),
\end{equation}
\noindent where $C>0$. As it is customary when dealing with functional inequalities, it is interesting to detect whether there is an optimal constant for which \eqref{eq:NonOptimalHardy} holds and whether this is ever achieved. These questions are nowadays classical for the inequality \eqref{eq:NonOptimalHardy}: the best constant is purely dimensional, it is given by 
$$\Lambda_{n}:= \dfrac{(n-2)^2}{4},$$
\noindent and it is never achieved. Similarly, there exists an optimal positive constant $\overline{\Lambda}_{n,s}$, depending only on $n$ and $s$ and never achieved, such that
\begin{equation}\label{eq:NonLocalOptimalHardy}
\overline{\Lambda}_{n,s} \, \int_{\mathbb{R}^{n}}\dfrac{u^2}{|x|^{2s}}\, dx \leq  
\frac{C_{n,s}}{2}[u]^2_s, \quad u\in C^{\infty}_{0}(\mathbb{R}^{n}),
\end{equation}
\noindent where 
$$[u]^2_{s}:= \iint_{\mathbb{R}^{2n}}\dfrac{|u(x)-u(y)|^2}{|x-y|^{n+2s}}\, dx\, dy,$$
\noindent denotes the Gagliardo seminorm of $u$ (of order $s \in (0,1)$).\\
The inequalities \eqref{eq:NonOptimalHardy} with $C=\Lambda_n$ and \eqref{eq:NonLocalOptimalHardy} are the starting point of our mixed local-nonlocal analysis. Clearly, the following trivially hold
\begin{equation*}
\Lambda_{n} \, \int_{\mathbb{R}^{n}}\dfrac{u^2}{|x|^2}\, dx \leq  \int_{\mathbb{R}^{n}}|\nabla u|^2\, dx + \frac{C_{n,s}}{2}[u]^2_s, \quad u\in C^{\infty}_{0}(\mathbb{R}^{n}),
\end{equation*}
\noindent and
\begin{equation*}
\overline{\Lambda}_{n,s} \, \int_{\mathbb{R}^{n}}\dfrac{u^2}{|x|^{2s}}\, dx \leq  \int_{\mathbb{R}^{n}}|\nabla u|^2\, dx + \frac{C_{n,s}}{2}[u]^2_s, \quad u\in C^{\infty}_{0}(\mathbb{R}^{n}).
\end{equation*}
As probably expected, the {\it best} Hardy potential (in a sense to be specified later on, see after Proposition \ref{prop:bestHardy}) to be considered is the one coming from the local Hardy inequality. This leads to the following natural question: does the minimization problem
\begin{equation}\label{eq:MinimizationProblem}
\Lambda_{s,n}:=\inf_{u \in C^{\infty}_{0}(\mathbb{R}^{n})\setminus \{0\}} \dfrac{\rho(u)^2}{\mathcal{H}(u)}
\end{equation}
\noindent admit a minimizer? Here, we have introduced the shorthand notation
\begin{equation*}
\rho(u)^2:=\|\nabla u\|^2_{L^{2}(\mathbb{R}^{n})}+\frac{C_{n,s}}{2}[u]^2_s \quad \mathrm{and} \quad \mathcal{H}(u):=\int_{\mathbb{R}^{n}}\dfrac{u^2}{|x|^2}\, dx.
\end{equation*}
Our first result answers to the above question.
\begin{thm}\label{thm:BestConstant}
The number defined by the minimization problem \eqref{eq:MinimizationProblem} coincides with the optimal local Hardy constant $\Lambda_n$ and it is never achieved.
\end{thm}
\medskip
Having in mind applications to mixed local-nonlocal PDEs, we are also interested in studying Hardy-type inequalities in open and bounded sets $\Omega \subset \mathbb{R}^n$. 
The second main result of the present paper is the following
\begin{thm}\label{thm:Main2}
Let $\Omega \subset \mathbb{R}^{n}$ be an open and bounded set such that $0\in \Omega$. Then
\begin{equation*}
\Lambda_{s,n}(\Omega):=\inf \left\{\rho(u)^2: u \in C^{\infty}_{0}(\Omega)\, \textrm{ s.t. } \,\mathcal{H}(u)=1\right\},
\end{equation*}
\noindent is independent of $\Omega$, coincides with $\Lambda_n$ and it is never achieved.
\end{thm}
In light of similar results proved for mixed local-nonlocal Sobolev inequalities (see \cite{BDVV5}), the above fact is not so surprising and it is motivated by the lack of scaling invariance of the functional $\rho(u)$.

\medskip

The aforementioned results open the way to the study of mixed local-nonlocal PDEs in presence of the classical Hardy potential, namely
\begin{equation}\label{eq:DirichletProblem}{\tag {$\mathrm{D}_{\gamma,f}$}}
\left\{ \begin{array}{rl}
-\Delta u + (-\Delta)^s u - \gamma \dfrac{u}{|x|^2} = f & \mathrm{in } \,\Omega,\\
u=0 & \mathrm{in } \, \mathbb{R}^{n}\setminus \Omega,
\end{array}\right.
\end{equation}
\noindent where $n \geq 3$, $\Omega \subset \mathbb{R}^{n}$ is an open and bounded set with smooth enough boundary and such that $0\in \Omega$, $f\in L^{m}(\Omega)$ for some $m \geq 1$ and $\gamma \in (0,\Lambda_n)$.\\  
The {\it boundary condition} is the purely nonlocal one and it is quite natural if one is interested in Dirichlet-type boundary conditions, see, e.g., \cite{BDVV} for examples in the case $\gamma =0$.\\
Our aim is to study the regularity of the solutions of \eqref{eq:DirichletProblem} in terms of the summability of the source term $f$: this forces to consider different notions of solutions and it is nowadays quite classical. This line of research is quite old in the case of variational local elliptic operators as well as purely nonlocal ones. We refer to Subsection \ref{subsec:Dirichlet} for a detailed review of the known results in the case of $-\Delta$, of $(-\Delta)^s$ and $\mathcal{L}$. \\
Having this background in mind, we state our first result which provides existence, optimal solvability and gain of integrability of solutions which can be considered of {\it energy-type} since $f \in L^{m}(\Omega)$ with $m\geq \tfrac{2n}{n+2}$.

\begin{rem}
Here, and in all the paper, when we write $f \geq 0$, we mean that $f\gneq 0$ in $\Omega$, i.e.~$f$ is not identically zero. Moreover, we point out that the positivity assumption on $f$ is not a restriction, since our problem is linear. 
\end{rem}

\begin{thm}\label{thm:EnergySolutions}
Let $f\in L^m(\Omega)$ with $m\geq (2^*)'=\tfrac{2n}{n+2}$. Then, 
problem \eqref{eq:DirichletProblem} admits a unique weak solution $u_f\in\mathcal{X}^{1,2}(\Omega)$ (see Section \ref{sec.Prel} for the relevant definition), 
in the following sense: 
\begin{equation}  \label{eq:uniquesol}
	\begin{split}
	\int_{\R^n}\nabla u_f\cdot\nabla v\,dx\, &
	+ \frac{C_{n,s}}{2} \iint_{\R^{2n}}\frac{(u_f(x)-u_f(y))(v(x)-v(y))}{|x-y|^{n+2s}}\,dx\,dy - \gamma \int_{\mathbb{R}^n}\dfrac{u_f v}{|x|^2} \, dx \\
	&= \int_{\Omega}fv \, dx, \quad \textrm{for every } v \in \mathcal{X}^{1,2}(\Omega).
	\end{split}
\end{equation}
Moreover, if $f \geq 0$ then $u>0$ in $\Omega$. In addition, if  $f$ also belongs to $L^{m}(\Omega)$  with
$\tfrac{2n}{n+2}\leq m < \tfrac{n}{2}$ and
$$0< \gamma < \gamma(m):= \dfrac{n(m-1)(n-2m)}{m^2},$$
\noindent then the unique weak solution $u_f\in\mathcal{X}^{1,2}(\Omega)$ of
problem \eqref{eq:DirichletProblem}  is such that
$$u_f \in L^{m^{**}}(\Omega),\quad \textrm{where } m^{**}:= \tfrac{n m}{n-2m}.$$
\end{thm}

\medskip

The next result deals with a source term $f \in L^{m}(\Omega)$ with $1<m<\tfrac{2n}{n+2}$. In this case variational arguments are forbidden but one can rely on the technique leading to {\it approximated solutions} by means of suitable truncation arguments.

\begin{thm}\label{thm:W1m*}
Let $f\in L^{m}(\Omega)$ with $1<m<\tfrac{2n}{n+2}$ be a positive function. If
$$0< \gamma < \gamma(m):= \dfrac{n(m-1)(n-2m)}{m^2},$$
\noindent then there exists a positive distributional  solution  $u_f \in W_{0}^{1,m^{*}}(\Omega)$ of \eqref{eq:DirichletProblem}, where $m^{*}:= \tfrac{n m}{n-m}$: namely  $u_f/|x|^2 \in L^1(\Omega)$ and it holds
\begin{equation} \label{eq:del_thm_1.5}
	\begin{split}
		\int_{\R^n}\nabla u_f\cdot\nabla v\,dx\, &
		+ \frac{C_{n,s}}{2} \iint_{\R^{2n}}\frac{(u_f(x)-u_f(y))(v(x)-v(y))}{|x-y|^{n+2s}}\,dx\,dy - \gamma \int_{\mathbb{R}^n}\dfrac{u_f v}{|x|^2} \, dx \\
		&= \int_{\Omega}fv \, dx, \quad \textrm{for every } v \in C^\infty_0(\Omega).
	\end{split}
\end{equation}
\end{thm}

\begin{rem}
	As we will see in the proof of Theorem \ref{thm:W1m*}, the solution $u_f$ is obtained as the limit of solutions of the regularized problems \eqref{eq:Ausiliario}.  We also prove that such a type of solution (usually referred to as SOLA solution) is unique. Moreover using the fact that $u_f \in W_{0}^{1,m^{*}}(\Omega)$, by a density argument it can be shown that \eqref{eq:del_thm_1.5} holds in a stronger sense, that is $u_f \in W_{0}^{1,m^{*}}(\Omega)$ and 
it holds
\begin{equation*} 	\begin{split}
		\int_{\R^n}\nabla u_f\cdot\nabla v\,dx\, &
		+ \frac{C_{n,s}}{2} \iint_{\R^{2n}}\frac{(u_f(x)-u_f(y))(v(x)-v(y))}{|x-y|^{n+2s}}\,dx\,dy - \gamma \int_{\mathbb{R}^n}\dfrac{u_f v}{|x|^2} \, dx \\
		&= \int_{\Omega}fv \, dx, \quad \textrm{for every } v \in W_0^{1,(m^*)'}(\Omega),
	\end{split}
\end{equation*}
where $(m^*)'=m^*/(m^*-1)$. 
\end{rem}

\medskip 

The last result concerns the case of $f \in L^{1}(\Omega)$, where we cannot ensure (in general) that $f \in \mathbb{X} = \left(\mathcal{X}^{1,2}(\Omega)\right)'$, and thus variational arguments are once again forbidden. As for the local case \cite{AP, PeralBook}, we are able to provide an optimal solvability criteria. Nevertheless, due to the nature of $\mathcal{L}$, which makes the task of finding explicit solutions quite hard, our condition is somehow implicit. In order to give the existence result, we need to consider solutions of an auxiliary problem. 
Let $g \in L^\infty (\Omega)$ and let $w \in \mathcal{X}^{1,2}(\Omega)$ be the unique weak solution of 
\begin{equation}\label{eq:DirichletProblemaux}
	\left\{ \begin{array}{rl}
		-\Delta w + (-\Delta)^s w= g & \mathrm{in } \,\Omega,\\
		w=0 & \mathrm{in } \, \mathbb{R}^{n}\setminus \Omega,
	\end{array}\right.
\end{equation}
in the sense of \eqref{eq:uniquesol}, with $\gamma=0$. 	We note that the unique solution $w$ belongs to $L^\infty(\Omega)$. Now, we introduce the notion of solution in this setting, adapting the classical concept of duality.

\begin{defn}\label{defin:duality}
	Let $f \in L^1(\Omega)$, and let $u\in L^1(\Omega)$. We say that
	$u$ is a duality solution of \eqref{eq:DirichletProblem}, if the following conditions hold:
	\begin{enumerate}
		\item $u\equiv 0$ a.e.\,in $\R^n\setminus\Omega$ and $u/ |x|^2 \in L^1(\Omega)$;
		
		\item for every $g \in L^\infty (\Omega)$ and $w$ solution to  \eqref{eq:DirichletProblemaux}, we have
			\begin{equation*} 
			\begin{split}
				\int_{\Omega}g u \, dx=\gamma \int_{\Omega} \dfrac{u}{|x|^2} w \, dx  + \int_{\Omega} f w \, dx.
			\end{split}
		\end{equation*}
	\end{enumerate}

\end{defn}

\begin{rem}
	We notice that any function $u$ which solves \eqref{eq:DirichletProblem} in the sense of Theorems \ref{thm:EnergySolutions} and \ref{thm:W1m*} is also duality solutions 
	(in the sense of Definition \ref{defin:duality}).
\end{rem}
With Definition \ref{defin:duality} at hand, we can now state the \emph{optimal-solvability result}
for problem \eqref{eq:DirichletProblem}.
Before doing this, we first fix a notation: \emph{we denote by $\Phi_\Omega$ the unique
	weak solution} in $\mathcal{X}^{1,2}(\Omega)$ of the following problem
\begin{equation*} 
	\begin{cases}
		\displaystyle \LL u = \gamma\frac{u}{|x|^2}+1 & \text{in $\Omega$}, \\
		u = 0 &\text{in $\R^n\setminus\Omega$},
	\end{cases}
\end{equation*} 
We explicitly observe that, since the constant function $f\equiv 1$ is positive
and bounded on $\Omega$, the existence and uniqueness of $\Phi_\Omega$
follow from Theorem \ref{thm:EnergySolutions}; moreover, we also have 
$$\text{$\Phi_\Omega > 0$ a.e.\,in $\Omega$}.$$
\begin{thm} \label{thm:solvabilityL1main}
Let $f\in L^1(\Omega),\,\text{$f\geq 0$ a.e.\,in $\Omega$}$. Then, there exists a 
 positive duality solution $u \in L^1(\Omega)$ of problem \eqref{eq:DirichletProblem} 
 in the sense of Definition \ref{defin:duality} if and only if
 \begin{equation} \label{eq:conditionfPhi}
 	\int_\Omega f\Phi_\Omega\,dx < \infty.
 \end{equation}
 In this case, one can also prove that $u$ satisfies the following properties:
 \begin{itemize}
    \item[{1)}] for every $k\in\mathbb{N}$, one has $T_k(u)\in \mathcal{X}^{1,2}(\Omega)$ (see \eqref{eq:defTk} for the definition of $T_k$); 
     \item[{2)}] $u|_\Omega$ is uniformly bounded in $W_0^{1,p}(\Omega)$ for every $p < \frac{n}{n-1}$.
    \vspace*{0.1cm}
 \end{itemize}
\end{thm}

\medskip

\indent The paper is organized as follows: in Section \ref{sec.Prel} we firstly describe the natural functional setting associated with $\mathcal{L}$ (see Subsection \ref{subsec:setting}), we then review all the regularity results connected with our problem (see Subsection \ref{subsec:Dirichlet}) and then we prove both Theorem \ref{thm:BestConstant} and Theorem \ref{thm:Main2} (see Subsection \ref{subsec:MixedHardy}). In Section \ref{sec:PDE} we move to the study of \eqref{eq:DirichletProblem} proving Theorem \ref{thm:EnergySolutions}, Theorem \ref{thm:W1m*} and Theorem \ref{thm:solvabilityL1main}.

\section{The functional setting and Preliminary results}\label{sec.Prel} 
  The aim of this section is to introduce the adequate functional setting
  for the study of the mixed operator $\LL = -\Delta+(-\Delta)^s$, and
  to collect some known results concerning the $\LL$-Dirichlet problem
  \begin{equation}\label{eq:DirPb}\tag{{$\mathrm{D}_f$}}
   \begin{cases}
  \LL u = f & \text{in $\Omega$}, \\
  u \equiv 0 & \text{in $\R^n\setminus\Omega$}
  \end{cases}
  \end{equation}
  where $\Omega\subseteq\R^n$ is a bounded domain and $f\in L^m(\Omega)$ for some $m\geq 1$. We also study mixed Hardy type inequality, by proving Theorem \ref{thm:BestConstant} and \ref{thm:Main2}.
  \vspace*{0.1cm}
  
  \subsection{The functional setting}\label{subsec:setting}
Let~$s\in (0,1)$ be fixed, and let
$\varnothing\neq \Omega\subseteq\R^n$ (with $n\geq 3$) be an \emph{arbitrary} open set,
  not necessarily bounded. We define the function space
  $\mathcal{X}^{1,2}(\Omega)$ as the com\-ple\-tion
  of $C_0^\infty(\Omega)$, that is the set of smooth functions with compact support in $\Omega$, with respect to the
  \emph{mixed global norm} 
  $$\rho(u) := \left(\|\nabla u\|^2_{L^2(\R^n)}+\frac{C_{n,s}}{2}[u]^2_s\right)^{1/2},\qquad u\in C_0^\infty(\Omega),$$
  where $[u]_s$ denotes the so-called \emph{Gagliardo seminorm} of $u$ (of order $s$), that is,
  $$[u]_s := \bigg(\iint_{\R^{2n}}\frac{|u(x)-u(y)|^2}{|x-y|^{n+2s}}\,dx\,dy\bigg)^{1/2}.$$
  \begin{rem} \label{rem:X12prop}
   A couple of observations concerning the space $\mathcal{X}^{1,2}(\Omega)$ are in order.
   \begin{enumerate}
    \item The norm $\rho(\cdot)$ is induced by the scalar product (or bilinear form)
    \begin{equation*}
    \mathcal{B}(u,v) := \int_{\R^n}\nabla u\cdot\nabla v\,dx
    + \frac{C_{n,s}}{2} \iint_{\R^{2n}}\frac{(u(x)-u(y))(v(x)-v(y))}{|x-y|^{n+2s}}\,dx\,dy,
    \end{equation*}
     where $\cdot$ denotes the usual scalar product in the Euclidean space
    $\R^n$, while 
    $$C_{n,s}:=2^{2s}\pi^{-\frac{n}{2}}\Gamma((n+2s)/2)/|\Gamma(-s)|;$$ 
    thus, $\mathcal{X}^{1,2}(\Omega)$
    is endowed with a structure of a (real) \emph{Hilbert space}.
    \medskip
    
    \item Even if a function $u\in C_0^\infty(\Omega)$ 
  \emph{identically vanishes outside~$\Omega$}, it is often still convenient to consider 
  in the definition of $\rho(\cdot)$ the $L^2$-norm of $\nabla u$
  \emph{on the whole of $\R^n$}, rather than restricted to~$\Omega$ 
  {(}though of course the result would be the same{)}: this is to stress that the elements in
  $\mathcal{X}^{1,2}(\Omega)$
  are functions defined \emph{on the entire space $\R^n$} and not only on~$\Omega$ 
  (and this is consistent with the nonlocal
  nature of the operator $\LL$). The benefit of having this global functional setting
  is that these functions can be \emph{globally approximated on $\R^n$} 
  {(}with respect to the norm $\rho(\cdot)${)}
  by smooth functions with compact support in $\Omega$.

  \noindent
  In particular, when $\Omega\neq \R^n$, we will
  see that this \emph{global} definition of 
  $\rho(\cdot)$ implies that the functions in $\mathcal{X}^{1,2}(\Omega)$ naturally
  satisfy the nonlocal Dirichlet condition 
    \begin{equation} \label{eq:nonlocalDirX12}
   \text{$u\equiv 0$ a.e.\,in $\R^n\setminus\Omega$ for every $u\in\mathcal{X}^{1,2}(\Omega)$}.
   \end{equation}
   \end{enumerate}
  \end{rem}
  In order to better understand the \emph{nature} of the space
  $\mathcal{X}^{1,2}(\Omega)$ (and to recognize the validity of
  \eqref{eq:nonlocalDirX12}), we distinguish two cases.
  \medskip
  
  (i)\,\,If \emph{$\Omega$ is bounded}. In this case we first recall the following
  inequality, which
   expresses the \emph{continuous embedding}
   of $H^1(\R^n)$ into $H^s(\R^n)$ (see, e.g., \cite[Proposition~2.2]{DRV}):
   there exists a constant $\mathbf{c} = \mathbf{c}(n,s) > 0$ such that,
   for every $u\in C_0^\infty(\Omega)$, one has
   \begin{equation} \label{eq:embeddingH1Hs}
    [u]^2_s \leq \mathbf{c}(n,s)\|u\|_{H^1(\R^n)}^2 = \mathbf{c}(n,s)\big(\|u\|^2_{L^2(\R^n)}
   + \|\nabla u\|^2_{L^2(\R^n)}\big).
   \end{equation}
   This, together with the classical \emph{Poincar\'e inequality}, implies that
   $\rho(\cdot)$ and the full $H^1$-norm in $\R^n$
   are \emph{actually equivalent} on the space $C^\infty_0(\Omega)$, and hence
   \begin{align*}
    \mathcal{X}^{1,2}(\Omega) & = \overline{C_0^\infty(\Omega)}^{\,\,\|\cdot\|_{H^1(\R^n)}} \\
    & = \{u\in H^1(\R^n):\,\text{$u|_\Omega\in H_0^1(\Omega)$ and 
    $u\equiv 0$ a.e.\,in $\R^n\setminus\Omega$}\}.
   \end{align*}
   
   (ii)\,\,If \emph{$\Omega$ is unbounded}. In this case, even if the \emph{embedding inequality}
   \eqref{eq:embeddingH1Hs} is still satisfied, the Poincar\'e inequality
   \emph{does not hold}; hence, the norm $\rho(\cdot)$ is no more equivalent to the full $H^1$-norm in $\R^n$,
   and $\mathcal{X}^{1,2}(\Omega)$ \emph{is not a subspace of $H^1(\R^n)$}.
   
   On the other hand, by the classical \emph{Sobolev inequality}, there exists 
   a constant $\mathcal{S} = \mathcal{S}_n > 0$, 
   \emph{independent of the open set $\Omega$}, such that
   \begin{equation} \label{eq:Sobolevmista}
    \mathcal{S}_n\|u\|^2_{L^{2^*}(\R^n)} \leq \|\nabla u\|^2_{L^2(\R^n)}
    \leq \rho(u)^2\qquad\text{for every $u\in C_0^\infty(\Omega)$}.
   \end{equation}   
   As a consequence, we deduce that
   $\mathcal{X}^{1,2}(\Omega)\hookrightarrow L^{2^*}(\R^n).$
   \begin{rem} \label{rem:XembeddedLtwostar}
   We explicitly notice that, since the \emph{mixed
   Sobolev-type inequality} \eqref{eq:Sobolevmista} holds \emph{for every open set $\Omega\subseteq\R^n$}
   (bounded or not), we always have
   \begin{equation} \label{eq:contEmbXL2star}
    \mathcal{X}^{1,2}(\Omega)\hookrightarrow L^{2^*}(\R^n).
   \end{equation}
   Furthermore, by exploiting
   the density of $C_0^\infty(\Omega)$ in $\mathcal{X}^{1,2}(\Omega)$,
   we can e\-xtend
   i\-ne\-qua\-lity
   \eqref{eq:Sobolevmista} to \emph{every function $u\in \mathcal{X}^{1,2}(\Omega)$}, thereby obtaining
   $$\mathcal{S}_n\|u\|_{L^{2^*}(\R^n)}^2 
    \leq \rho(u)^2 = \|\nabla u\|_{L^2(\R^n)}^2
    + \frac{C_{n,s}}{2}[u]^2_s\quad\text{for every $u\in\mathcal{X}^{1,2}(\Omega)$}.$$
   \end{rem}
   Due to its relevance in the sequel, we also introduce a distinguished notation for the cone 
   of the non-negative functions in $\mathcal{X}^{1,2}(\Omega)$: we set
   $$\mathcal{X}^{1,2}_+(\Omega) = 
   \big\{u\in\mathcal{X}^{1,2}(\Omega):\,\text{$u\geq 0$ a.e.\,in $\Omega$}\big\}.$$
   
\subsection{The $\LL$-Dirichlet problem} \label{subsec:Dirichlet}
Now we have introduced
the space $\mathcal{X}^{1,2}(\Omega)$, we present
some known results concerning \emph{existence and optimal regularity of solutions}
for the $\LL$-Di\-richlet problem \eqref{eq:DirPb}. To this end, we distinguish two different cases.
\begin{align*}
\mathrm{i)}&\,\,\text{$f\in L^m(\Omega)$ for some $m\geq (2^*)' = \frac{2n}{n+2}$}; \\[0.1cm]
\mathrm{ii)}&\,\,\text{$f\in L^m(\Omega)$ for some $1<m<(2^*)'$}.
\end{align*}
\textbf{Case i) - The variational case.} In this case we observe that, by combining the \emph{co\-n\-ti\-nuo\-us embedding} \eqref{eq:contEmbXL2star} of $\mathcal{X}^{1,2}(\Omega)$
with
the H\"older inequality, we have
$$\int_\Omega |fv|\,dx <+\infty\quad\forall\,\,v\in\mathcal{X}^{1,2}(\Omega),$$
and thus $f\in \mathbb{X} = (\mathcal{X}^{1,2}(\Omega))'$;
in view of this fact, \emph{existence and uniqueness} of solutions for problem 
\eqref{eq:DirPb} easily follow
from the Lax-Milgram Theorem. More precisely, we present the following theorem, which collects most of the published results regarding the existence, uniqueness and regularity of solutions to problem \eqref{eq:DirPb}.
   \begin{thm} \label{thm:generalDirichlet}
  Let $\Omega\subseteq\R^n$ be a bounded open set with smooth boundary,
  and let $f\in L^m(\Omega)$ for some $m\geq (2^*)'$. 
  Then, \emph{there exists a unique weak solution} $u_f\in\mathcal{X}^{1,2}(\Omega)$ of 
  problem \eqref{eq:DirPb}, 
  in the following sense:
  \begin{equation}\label{eq:defSolVariational}
   \mathcal{B}(u_f,v) = \int_\Omega fv\,dx\qquad\forall\,\,v\in\mathcal{X}^{1,2}(\Omega).
  \end{equation}
  Furthermore, the following properties hold. 
  \begin{itemize}
   \item[i)] \emph{(}see \cite[Theorem\,4.7]{BDVV}\emph{)} If $f\in L^m(\Omega)$ with 
   $m > n/2$, then $u_f\in L^\infty(\R^n)$.
   \vspace*{0.05cm}
   
   \item[ii)] \emph{(}see \cite[Theorem\,2.3]{ArRa}\emph{)} If $(2^*)'\leq m \leq n/2$, 
   then $u_f \in L^{m^{**}}(\Omega)$, where
	$$m^{**} = \frac{nm}{n-2m}.$$
	
 \item[iii)] \emph{(}see \cite[Theorem\,1.1]{AntoCozzi}\emph{)} 
   If $f\in L^m(\Omega)$ with $m > n$, then 
   $$\text{$u_f\in C^{1,\theta}(\overline{\Omega})$ 
   for some
 $\theta\in (0,1)$}.$$
 
 \item[iv)] \emph{(}see \cite[Theorem\,B.1]{BDVV3}\emph{)} If $f\in C^\alpha(\overline{\Omega})$
 for some $\alpha\in (0,1)$, and if $2s+\alpha < 1$ \emph{(}hence, in particular, $0<s<1/2$\emph{)}, then
 $u_f\in C^{2,\alpha}(\overline{\Omega})$.
 \vspace*{0.05cm}
 
\item[v)] \emph{(}see, e.g., \cite[Theorem\,1.2]{BDVV} 
and \cite[Corollary\,3.3 and Rem.\,3.4]{BMV} \emph{)}
 If $f\geq 0$ a.e.\,in $\Omega$, then we have
 $u_f\geq 0$ a.e.\,in $\Omega$; moreover,
 either
 $u_f\equiv 0$ or $u_f > 0$ in $\Omega$.
\end{itemize}
 \end{thm}
  
 \noindent \textbf{Case ii) - The nonvariational case.} In this case, the inclusion $L^m(\Omega)\subseteq
 \mathbb{X}=(\mathcal{X}^{1,2}(\Omega))'$ \emph{does not hold}, and we cannot directly 
 apply the Lax-Milgram Theorem
 to study the e\-xi\-stence of solutions for problem \eqref{eq:DirPb}. As in the purely local setting, 
 \emph{existence and improved regularity} are proved at the same time
 by an approximation argument.
 \begin{thm}[{See \cite[Theorem 2.1]{ArRa}}] \label{thm:ArRaNonVar}
  Let $\Omega\subseteq\R^n$ be a bounded open set with smooth boundary,
  and let $f\in L^m(\Omega)$ for some $1<m< (2^*)'$, $f\geq 0$ a.e.\,in $\Omega$.
  Then, there exists
  a positive solution $u_f$ of problem \eqref{eq:DirPb}, in the following sense
  \begin{itemize}
   \item[a)] $u_f\in W_0^{1,m^*}(\Omega)$, with $m^* = nm/(n-m)$;
   \vspace*{0.05cm}
   
   \item[b)] for every $v\in W_0^{1,(m^*)'}(\Omega)\cap L^{m'}(\Omega)$ with \emph{compact support}, we have
   $$\mathcal{B}(u_f,v) = \int_\Omega fv\,dx.$$
\end{itemize}
 \end{thm}
 \begin{rem}[A comparison with the local/nonlocal case] \label{rem:compareLocalNonlocal}
 It is interesting to notice that, on account of Theorems \ref{thm:generalDirichlet}-\ref{thm:ArRaNonVar},
  the \emph{improved integrability} for the solutions of \eqref{eq:DirPb} is generated only by $-\Delta$, thus reflecting the merely perturbative role of $(-\Delta)^s$, at least when working on bounded sets. More precisely, let $u_f$ be the solution (to be specified, depending on the summability of $f$) of 
  $$\begin{cases}
   -\Delta u = f\in L^m(\Omega) & \text{in $\Omega$} \\
   u = 0 & \text{in $\de\Omega$}.
   \end{cases}.$$
\noindent Then, in \cite{Stampacchia} where more general operators are considered, Stampacchia showed that
  \begin{itemize}
\item[i)] if $1<m<(2^*)'$ then  $u_f \in W^{1,m^*}_{0}(\Omega)$;
\item[ii)] if $(2^*)'\leq m\leq n/2$ then $u_f \in H^{1}_{0}(\Omega)\cap L^{m^{**}}(\Omega)$;
\item[iii)] if $m >n/2$ then $u_f\in H^{1}_{0}(\Omega)\cap L^{\infty}(\Omega)$.
\end{itemize}
A similar results holds true in the purely nonlocal framework, that is,
for the solution $u_f$ of the following \emph{non-local Dirichlet problem}
$$\begin{cases}
   (-\Delta)^s u = f\in L^m(\Omega) & \text{in $\Omega$} \\
   u = 0 & \text{in $\R^n\setminus\Omega$}
   \end{cases}.$$
In this case, setting
$$m^{*}_{s}:= \dfrac{n m}{n-ms} \quad \textrm{and} \quad m^{**}_s:=\dfrac{nm}{n-2ms},$$
\noindent we have that (see \cite[Theorem 13, Theorem 16, Theorem 17, Theorem 24]{LPPS})
\begin{itemize}
\item[a)] if $1<m<{2n}/(n+2s)$ then the weak-duality solution $u_f \in L^{m^{**}_{s}}(\Omega)$;
\item[b)] if $2n/(n+2s)\leq m\leq n/(2s)$ then $u_f \in H^{s}_{0}(\Omega)\cap L^{m^{**}_{s}}(\Omega)$;
\item[c)] if $m >n/(2s)$ then $u\in H^{s}_{0}(\Omega)\cap L^{\infty}(\Omega)$.
\end{itemize}
 \end{rem}

\subsection{Mixed Hardy-type inequalities}\label{subsec:MixedHardy}
Now, we focus our attention on the mixed Hardy inequalities. In order to do this, we need to introduce some notations and recall some well known results. Let  $n \geq 3$. The classical Hardy inequality states that 
\begin{equation*}
 \frac{(n-2)^2}{4}\int_{\mathbb{R}^n}\dfrac{u^2}{|x|^2}\, dx \leq \int_{\mathbb{R}^n}|\nabla u|^2 \, dx, \quad u\in C^{\infty}_{0}(\mathbb{R}^n),
\end{equation*}
\noindent and it is known that this constant $\Lambda_n = (n-2)^2/4$ is never achieved, see e.g. \cite{PeralBook}. 
Clearly,
\begin{equation}\label{eq:HardyMix}
  \Lambda_n \int_{\mathbb{R}^n}\dfrac{u^2}{|x|^2}\, dx \leq \int_{\mathbb{R}^n}|\nabla u|^2 \, dx \leq \rho(u)^2, \quad u\in C^{\infty}_{0}(\mathbb{R}^n).
\end{equation}
In particular, due to the mixed nature of the norm $\rho(u)$, it holds
\begin{equation}\label{eq:HardyMixFractional}
  \overline{\Lambda}_{n,s} \int_{\mathbb{R}^n}\dfrac{u^2}{|x|^{2s}}\, dx  \leq \rho(u)^2, \quad u\in C^{\infty}_{0}(\mathbb{R}^n),
\end{equation}
\noindent as well, where $\overline{\Lambda}_{n,s}$ denotes the optimal constant for the fractional Hardy inequality, see e.g. \cite[Chapter 9]{PeralBook}. Combining \eqref{eq:HardyMix} and \eqref{eq:HardyMixFractional}, we get the following
\begin{prop}\label{prop:bestHardy}
For every $p\in [2s,2]$, there exists a positive constant $C>0$, depending only on $n$ and $s$, such that
\begin{equation}\label{eq:HardyMixFractional2}
  C \int_{\mathbb{R}^n}\dfrac{u^2}{|x|^{p}}\, dx  \leq \rho(u)^2, \quad u\in C^{\infty}_{0}(\mathbb{R}^n),
\end{equation}
\end{prop}
\begin{proof}
 Let $p\in[2s,2]$ be arbitrarily fixed. Owing to \eqref{eq:HardyMix}-\eqref{eq:HardyMixFractional}, we have
 \begin{align*}
  \int_{\R^n}\frac{u^2}{|x|^{p}}\, dx
  & = \int_{\{|x|\leq 1\}}\frac{u^2}{|x|^{p}}\, dx
  + \int_{\{|x| > 1\}}\frac{u^2}{|x|^{p}}\, dx \\
  & \leq \int_{\{|x|\leq 1\}}\frac{u^2}{|x|^{2}}\, dx
  + \int_{\{|x| > 1\}}\frac{u^2}{|x|^{2s}}\, dx \\
  & \leq \Big(\frac{1}{\Lambda_n}+\frac{1}{\overline{\Lambda}_{n,s}}\Big)\rho(u)^2\quad
  \forall\,\,u\in C_0^\infty(\R^n),
 \end{align*}
 and this is precisely the desired \eqref{eq:HardyMixFractional2}.
\end{proof}
We notice that for $p=2$ we can take $C = \Lambda_n$, while for $p=2s$ we can take $C=\overline{\Lambda}_{n,s}$; however, due to the monotonicity of $|x|^{-p}$ it seems that the most interesting case is the one for $p=2$.
Let us focus for a moment to the case $p=2$, and let us introduce the shorthand notation
$$\mathcal{H}(u):= \int_{\mathbb{R}^n}\dfrac{u^2}{|x|^2}\, dx.$$
It is natural to wonder whether the minimization problem
\begin{equation*}
\Lambda_{s,n}:= \inf_{u \in C^{\infty}_{0}(\mathbb{R}^n)\setminus \{0\}} \dfrac{\rho(u)^2}{\mathcal{H}(u)}
\end{equation*}
\noindent admits a minimizer, that is whether $\Lambda_{s,n}$ is achieved or not. The answer is given by Theorem \ref{thm:BestConstant} which we will prove right below:

\begin{proof}[Proof of Theorem \ref{thm:BestConstant}]
Let $u \in C^{\infty}_{0}(\mathbb{R}^n)$. By the very definition of $\rho(u)$, we have that
\begin{equation*}
\Lambda_n \leq \dfrac{\|\nabla u\|^2_{L^2(\mathbb{R}^n)}}{\mathcal{H}(u)}\leq \dfrac{\rho(u)^2}{\mathcal{H}(u)},
\end{equation*}
\noindent and therefore, taking the infimum, 
\begin{equation}\label{eq:LambdanLeq}
\Lambda_n \leq \Lambda_{s,n}.
\end{equation}
On the other hand, if $u \in C^{\infty}_{0}(\mathbb{R}^n)$, we have that (for $\lambda >0$),
$$u_{\lambda}(x):= \lambda^{\tfrac{n-2}{2}}u(\lambda x) \in C^{\infty}_{0}(\mathbb{R}^n),$$
\noindent and therefore it can be used as a test function, finding
\begin{equation*}
\begin{aligned}
\Lambda_{s,n} &\leq \dfrac{\rho(u_{\lambda})^2}{\mathcal{H}(u_{\lambda})} = \dfrac{\|\nabla u\|^2_{L^{2}(\mathbb{R}^n)} + \lambda^{2s-2} \frac{C_{n,s}}{2} [u]^2_s}{\mathcal{H}(u)} \\
&=\dfrac{\|\nabla u\|^2_{L^{2}(\mathbb{R}^n)}}{\mathcal{H}(u)} + \lambda^{2s-2} \frac{C_{n,s}}{2} \dfrac{[u]^2_s}{\mathcal{H}(u)} \to \dfrac{\|\nabla u\|^2_{L^{2}(\mathbb{R}^n)}}{\mathcal{H}(u)}, \quad \textrm{as } \lambda \to +\infty.
\end{aligned}
\end{equation*}
We can then conclude that
$$\Lambda_{s,n} \leq \Lambda_n,$$
\noindent which combined with \eqref{eq:LambdanLeq} proves $\Lambda_{s,n} = \Lambda_n$. The fact that the constant is never achieved follows from the upcoming Proposition \ref{thm:LambdaneverAch}.
\end{proof}

\medskip

Let us now move to the case of bounded sets. Firstly, let us define the set
\begin{equation*}
\mathcal{M}(\Omega):= \left\{ u \in C^{\infty}_{0}(\Omega): \mathcal{H}(u)=1\right\},
\end{equation*}
and the constants
\begin{equation*}
	\Lambda_{s,n}(\Omega):= \inf_{u \in C^{\infty}_{0}(\Omega)\setminus \{0\}} \dfrac{\rho(u)^2}{\mathcal{H}(u)} =  \inf\left\{ \rho(u)^2: u\in C^{\infty}_{0}(\Omega)\cap \mathcal{M}(\Omega)\right\}.
\end{equation*}

We now prove the following proposition leading to the demonstration of Theorem \ref{thm:Main2}.

\begin{prop} \label{thm:lambdanslambdan}
Let $\Omega \subseteq \mathbb{R}^n$ be an open and bounded set such that $0 \in \Omega$. Then
\begin{equation}
\Lambda_{s,n}(\Omega) = \Lambda_n.
\end{equation}
\begin{proof}
Since $\rho(u) \geq \|\nabla u\|_{L^2(\Omega)}$ for every $u \in C^{\infty}_{0}(\Omega)$, it holds that
\begin{equation*}
\Lambda_{s,n}(\Omega) \geq \Lambda_n. 
\end{equation*}
  To prove the reverse inequality, we consider $r > 0$ be such that $B_r(0)\subseteq\Omega$.
  We now observe that, given any $u\in C_0^\infty(\R^n)\cap \mathcal{M}(\R^n)$,
  there exists $k_0 = k_0(u)\in\N$ such that 
  $$\mathrm{supp}(u)\subseteq B_{kr}(0)\quad\text{for every $k\geq k_0$};$$
  as a consequence, setting $u_k(x) := k^{\frac{n-2}{2}}u(kx)$ (for $k\geq k_0$), we readily see that
  $$\mathrm{supp}(u_k)\subseteq B_r(0)\subseteq \Omega\quad\text{and}\quad
 \mathcal{H}(u_k) = 1.$$
  Taking into account the definition of $\Lambda_{s,n}(\Omega)$, we find that, for every $k\geq k_0$,
  \begin{align*}
   \Lambda_{s,n}(\Omega) & \leq \rho(u_k)^2 = 
   \|\nabla u_k\|^2_{L^2(\R^n)}+ \frac{C_{n,s}}{2} [u_k]^2_s
= \|\nabla u\|^2_{L^2(\R^n)}+k^{2s-2} \frac{C_{n,s}}{2} [u]^2_s .
  \end{align*}
  From this, letting $k\to\infty$ (and recalling that $s < 1$), we obtain
  $$\Lambda_{s,n}(\Omega)\leq \|\nabla u\|^2_{L^2(\R^n)}.$$
  By the arbitrariness of $u\in C^\infty_0(\R^n)\cap \mathcal{M}(\R^n)$ and the fact that
   $\Lambda_n$ \emph{is independent of the open set $\Omega$}, we finally infer that
   $$\Lambda_{s,n}(\Omega)\leq \inf\big\{\|\nabla u\|^2_{L^2(\R^n)}:\, u \in C^\infty_0(\R^n)\cap \mathcal{M}(\R^n)
   \big\} = \Lambda_n,$$
   and hence $\Lambda_{s,n}(\Omega) = \Lambda_n$.
\end{proof}
\end{prop}

\begin{prop} \label{thm:LambdaneverAch}
Let $\Omega \subseteq \mathbb{R}^n$ be an open set (not necessarily bounded). Then, $\Lambda_{s,n}(\Omega)$ is never achieved.
\end{prop}
\begin{proof}
Arguing by contradiction, let us suppose that there exists a
  function $u_0\in\mathcal{X}^{1,2}(\Omega)$
  such that 
  $\mathcal{H}(u_0) = 1$ (hence, $u_0\not\equiv 0$) and
  $$\rho(u_0)^2 = \|\nabla u_0\|^2_{L^2(\R^n)}+\frac{C_{n,s}}{2}[u_0]^2_s = \Lambda_{n}.$$
 Since $\mathcal{X}^{1,2}(\Omega)\subseteq\mathcal{D}^{1,2}(\Omega)$, we can then infer that
  $$\Lambda_n \leq \|\nabla u_0\|^2_{L^2(\Omega)} \leq \|\nabla u_0\|^2_{L^2(\R^n)}+\frac{C_{n,s}}{2}[u_0]^2_s
  = \rho(u_0)^2 = \Lambda_n,$$
  from which we derive that $[u_0]_s = 0$.
  As a consequence, the function 
  $u_0$ must be \emph{constant in $\R^n$}, but this is contradiction with the fact that
  $\mathcal{H}(u_0) = 1$.
\end{proof}

We are therefore ready to prove Theorem \ref{thm:Main2}.
\begin{proof}[Proof of Theorem \ref{thm:Main2}]
It is enough to combine Proposition \ref{thm:lambdanslambdan} with Proposition \ref{thm:LambdaneverAch}.
\end{proof}

 We end this section with the following lemma, showing
 a connection between the Hardy-type inequality 
 \eqref{eq:HardyMixFractional2} and the existence of solutions
 for the variational inequality
 $$\LL u\geq \gamma\frac{u}{|x|^2}\quad\text{in $\Omega$}.$$
 In order
 to state and prove such a result, we first fix a notation: 
 given any $k\in\mathbb{N}$, we set
 \begin{equation} \label{eq:defTk}
 	T_{k}(t):= \max \{\min\{k,t\},-k\} = \left\{ \begin{array}{lr}
 		t, & |t| \leq k,\\
 		k\dfrac{t}{|t|}, & |t| >k.
 	\end{array}\right.
\end{equation}	
 \begin{lem} \label{lem:Musina}
   Let $\Omega\subseteq\R^n$ be a bounded open set with smooth boundary, let $\gamma > 0$ be fixed,  and let $u\in\mathcal{X}^{1,2}(\Omega)$, $u>0$, be such that
   $$\mathcal{B}(u,v)\geq \gamma\int_\Omega\frac{uv}{|x|^2}\,dx\qquad\text{for all
   $v\in\mathcal{X}^{1,2}_+(\Omega)$}.$$
   Then, we have
   $$\rho(v)^2\geq \gamma\int_\Omega\frac{v^2}{|x|^2}\,dx\quad
   \text{for all $v\in\mathcal{X}^{1,2}(\Omega)$}.$$
 \end{lem}
 
 \begin{proof}
  This proof is inspired by \cite[Lem.\,B.1]{FallMus}. First of all, let us consider the following truncated problem given by
 \begin{equation}\label{eq:PT1}
 	\begin{cases}
 		\LL w = T_1 \left(\frac{u}{|x|^2}\right) & \text{in $\Omega$}\\
 		w=0 & \text{in $\R^n \setminus  \Omega$}.
 	\end{cases}
  \end{equation}
  We know that problem \eqref{eq:PT1} admits a unique positive weak solution $w \in \mathcal{X}^{1,2}(\Omega)$. Hence, noticing that 
  $$\LL u\geq \gamma\frac{u}{|x|^2} \geq \gamma T_1 \left(\frac{u}{|x|^2}\right) = \LL w \quad\text{in $\Omega$},$$
 and since the operator $\LL$ is linear, we easily deduce that
  \begin{equation*}
  	\begin{cases}
  		\LL (w-u) \leq 0 & \text{in $\Omega$}\\
  		w-u=0 & \text{in $\R^n \setminus  \Omega$}.
  	\end{cases}
\end{equation*}
  As a consequence, since $v = w-u\in\mathcal{X}^{1,2}(\Omega)\subseteq H^1(\R^n)$,
  we are entitled to apply the Weak Maximum Principle
  in \cite[Theorem\,1.2]{BDVV}, obtaining
  $$u \geq w \qquad \text{in } \R^n.$$
   In particular, since $w \in C^{1,\alpha}(\overline{\Omega})$ (see Theorem \ref{thm:generalDirichlet}), we deduce that for every compact set $\mathcal K \subset \Omega$
   there exists a constant $\vartheta >0$ such that
  \begin{equation} \label{eq:ugeqwWMP}
   u \geq w \geq \vartheta>0 \qquad \text{in } \mathcal{K},
   \end{equation}
  and this ensures that $u^{-1}\in L^\infty_{\mathrm{loc}}(\Omega)$. 
  
  Now we have established \eqref{eq:ugeqwWMP}, we can easily complete
  the proof of the lemma. Indeed, let
  $\varphi\in C_0^\infty(\Omega)$  be arbitrarily fixed, and let
  $$\psi = u^{-1}\varphi^2.$$ 
  By \eqref{eq:ugeqwWMP} (and since  supp$(\varphi) \subset \Omega$), we have $\psi\in L^\infty(\Omega)$; moreover,
  $$(u(x)-u(y))(\psi(x)-\psi(y)) \leq |\varphi(x)-\varphi(y)|^2.$$
  Indeed, by the very definition of $\psi$ we have
  \begin{equation}\label{eq:conto_su_psi}
  	\begin{split}
  	(u(x)-u(y))&(\psi(x)-\psi(y)) 	= \varphi(x)^2+\varphi(y)^2-u(x)\frac{\varphi^2(y)}{u(y)}-
  	u(y)\frac{\varphi^2(x)}{u(x)} \\
  	& = (\varphi(x)-\varphi(y))^2
  	+2\varphi(x)\varphi(y)-u(x)\frac{\varphi^2(y)}{u(y)}-
  	u(y)\frac{\varphi^2(x)}{u(x)} \\
  	& = (\varphi(x)-\varphi(y))^2
  	+2\varphi(x)\varphi(y)-u(x)u(y)\frac{\varphi^2(y)}{u^2(y)}-
  	u(x)u(y)\frac{\varphi^2(x)}{u^2(x)} \\
	& = (\varphi(x)-\varphi(y))^2-
  	u(x)u(y)\Big\{\frac{\varphi^2(x)}{u^2(x)}+\frac{\varphi^2(y)}{u^2(y)}
  	- 2\frac{\varphi(x)}{u(x)}\frac{\varphi(y)}{u(y)}\Big\} \\
  	& = (\varphi(x)-\varphi(y))^2-u(x)u(y) \Big(\frac{\varphi(x)}{u(x)} -\frac{\varphi(y)}{u(y)}\Big)^2  	\\
  	&\leq (\varphi(x)-\varphi(y))^2,
  \end{split}
  \end{equation}
  where we have also used the fact that $u > 0$ a.e.\,in $\Omega$.
  
  From now on we argue as in \cite{FallMus}. In order to prove the result we take $\psi$ as test function in
  $$\int_{\R^n} \nabla u \cdot \nabla \psi \, dx + \frac{C_{n,s}}{2}\iint_{\R^{2n}} \frac{(u(x)-u(y))(\psi(x)-\psi(y))}{|x-y|^{n+2s}} \,dx dy\geq \gamma\int_\Omega\frac{u \psi}{|x|^2}\,dx.$$
  Noticing that $\nabla u \cdot \nabla \psi
  = |\nabla\varphi|^2-u^2|\nabla(u^{-1}\varphi)|^2\leq |\nabla\varphi|^2$, and using \eqref{eq:conto_su_psi}, we deduce
  $$\int_{\R^n} |\nabla \varphi|^2 \, dx + \frac{C_{n,s}}{2} \iint_{\R^{2n}} \frac{(\varphi(x)-\varphi(y))^2}{|x-y|^{n+2s}} \,dx dy\geq \gamma\int_\Omega\frac{\varphi^2}{|x|^2}\,dx.$$
 \end{proof}

\section{The $\LL$-Dirichlet problem involving the Hardy potential}\label{sec:PDE}

Taking into account all the results recalled and/or
established so far, we now aim to study \emph{existence and improved-integrability}
of the solutions for the $\LL$-Dirichlet problem \eqref{eq:DirichletProblem} defined in the Introduction, that is
\begin{equation}\label{eq:HardyEigenvalue}{\tag {$\mathrm{D}_{\gamma,f}$}}
\left\{ \begin{array}{rl}
\mathcal{L}u - \gamma \dfrac{u}{|x|^2} = f & \textrm{in } \Omega,\\
u= 0 & \textrm{in } \mathbb{R}^n \setminus \Omega,
\end{array}\right.
\end{equation}
\noindent where $\Omega \subset \mathbb{R}^n$, $n \geq 3$, is an open, regular and  bounded set with $0 \in \Omega$, and
$$\text{$f \in L^m(\Omega)$ for some $m\geq 1$}\quad\text{and}\quad
\gamma \in \left(0, \Lambda_n\right).$$
To this end, motivated by the discussion in Section \ref{sec.Prel},
we distinguish the following three cases, according to the values of $m$.
\begin{align*}
\mathrm{i)}&\,\,\text{$f\in L^m(\Omega)$ for some $m\geq (2^*)' = \frac{2n}{n+2}$}; \\[0.1cm]
\mathrm{ii)}&\,\,\text{$f\in L^m(\Omega)$ for some $1<m<(2^*)'$}; \\[0.1cm] 
\mathrm{iii)}&\,\,f\in L^1(\Omega).
\end{align*}
\begin{rem} \label{rem:restriction}
 We explicitly stress that the restriction $0<\gamma<\Lambda_n$, generally, is motivated by
 the re\-sults in the previous section. Indeed, if we assume that there exists
 a \emph{positive weak solution} of problem \eqref{eq:HardyEigenvalue}
 for some $\lambda > 0$ and some $f\geq 0$, say $u_f\in\mathcal{X}^{1,2}_+(\Omega)$, then
 $$\mathcal{B}(u_f,v) = \gamma\int_{\Omega}\frac{u_fv}{|x|^2}\,dx
 +\int_\Omega fv\,dx\geq \gamma\int_{\Omega}\frac{u_fv}{|x|^2}\,dx
 \quad\text{for all $v\in\mathcal{X}^{1,2}_+(\Omega)$};$$
 this, together with Lemma \ref{lem:Musina} and Theorem \ref{thm:BestConstant}, implies that
 $$\Lambda_n
 = \inf_{u \in C^{\infty}_{0}(\Omega)\setminus \{0\}} \dfrac{\rho(u)^2}{\mathcal{H}(u)}
 \geq \gamma.$$
Hence, the last fact motivates nonexistence of positive energy solution of problem \eqref{eq:HardyEigenvalue} in the case $\gamma > \Lambda_{n}$ and $f$ sufficiently regular.

On the other hand, if we consider the case $\gamma = \Lambda_n$ in   \eqref{eq:HardyEigenvalue}, there is a different approach that can be borrowed, as pointed out in \cite{PeralBook}. Thanks to a more general version of the Hardy inequality, the improved Hardy inequality (see, e.g., \cite{VazZua}), one can define an Hilbert 
space $\mathcal{H}(\Omega)$ larger than $\mathcal{X}^{1,2}(\Omega)$ where it is possible to find a unique positive solution $u_f \in \mathcal{H}(\Omega)$ to \eqref{eq:HardyEigenvalue} if $f$ belongs to the dual space $\mathcal{H}'(\Omega)$ (e.g. when $f \in L^m(\Omega)$ with $m>2n/(n+2)$). 
\end{rem}

\subsection{Case i)\,-\,$f\in L^m(\Omega)$ with $m\geq (2^*)'$.}
As already pointed out in Section \ref{sec.Prel}, in this case we have $f\in\mathbb{X} =
(\mathcal{X}^{1,2}(\Omega))'$; as a consequence,
we can investigate existence and uniqueness of solutions
for \eqref{eq:DirPb} using the Lax-Milgram Theorem, obtaining the following result.
\begin{thm} \label{thm:LaxMilagrmVariational}
Let $f\in L^m(\Omega)$ with $m\geq (2^*)'$. Then, 
problem \eqref{eq:HardyEigenvalue} admits a unique weak solution $u_f\in\mathcal{X}^{1,2}(\Omega)$, 
in the following sense: 
\begin{equation} \label{eq:weakform}
\mathcal{B}(u,v) - \gamma \int_{\mathbb{R}^n}\dfrac{uv}{|x|^2} \, dx = \int_{\Omega}fv \, dx, \quad \textrm{for every } v \in \mathcal{X}^{1,2}(\Omega).
\end{equation}
In particular, if $f \equiv 0$, the unique weak solution is the trivial one; while, if $f\geq 0$ a.e.\,in $\Omega$, then $u_f \geq 0$ a.e.\,in $\Omega$. Indeed, $u_f>0$ a.e.\,in $\Omega$.
\end{thm}
\begin{proof}
The first assertion follows from Lax-Milgram Theorem: in fact, it suffices to notice that 
the bilinear form defined as
$$\mathcal{B}(u,v) - \gamma \int_{\mathbb{R}^n}\dfrac{uv}{|x|^2} \, dx$$
\noindent is coercive and continuous on the Hilbert space $\mathcal{X}^{1,2}(\Omega)$. 
The second assertion easily follows noticing that $u \equiv 0$ is a weak solution to \eqref{eq:HardyEigenvalue} when $f \equiv 0$. 

Finally, we turn to prove $u_f \geq 0$ a.e.\,in $\Omega$ when $f \geq 0$. To this end we first observe that,
since $u_f\in\mathcal{X}^{1,2}(\Omega)$, we have 
$$v := (u_f)^- = \max\{-u_f,0\}\in\mathcal{X}^{1,2}(\Omega).$$ 
We are then entitled to use
this  function $v$ as a test function in \eqref{eq:weakform}, obtaining
\begin{align*}
 0 & \leq \int_\Omega fv\,dx = \mathcal{B}(u_f,v) - \gamma \int_{\mathbb{R}^n}\dfrac{u_fv}{|x|^2} \, dx \\
 & = \int_{\R^n}\nabla u_f\cdot\nabla v\,dx
 +\frac{C_{n,s}}{2} \iint_{\R^{2n}}\frac{(u_f(x)-u_f(y))(v(x)-v(y))}{|x-y|^{n+2s}}\,dx\,dy
 - \gamma \int_{\mathbb{R}^n}\dfrac{u_fv}{|x|^2} \, dx \\
 & = -\Big(\int_\Omega |\nabla v|^2\,dx-\gamma\int_\Omega\frac{v^2}{|x|^2}\,dx\Big) 
 + \frac{C_{n,s}}{2} \iint_{\R^{2n}}\frac{(u_f(x)-u_f(y))(v(x)-v(y))}{|x-y|^{n+2s}}\,dx\,dy \\
  & (\text{since $(u_f(x)-u_f(y))(v(x)-v(y))\leq -|v(x)-v(y)|^2$}) \\
 & \leq -\Big(\rho(v)^2-\gamma\int_\Omega\frac{v^2}{|x|^2}\,dx\Big) \\
 & (\text{using Theorem \ref{thm:BestConstant}}) \\
 & \leq -\left(1-\frac{\gamma}{\Lambda_n}\right)\rho(v)^2
\end{align*}
From this, since $0<\gamma<\Lambda_n$, we conclude that $v = (u_f)^- = 0$ a.e.\,in $\Omega$, 
and thus
\begin{equation}\label{eq:u_f_geq_0}
u_f\geq 0 \quad \textrm{ a.e.\,in }\Omega.
\end{equation}
To show that $u_f > 0$ a.e.\,in $\Omega$ we now observe that, since $u_f$ is a weak solution of \eqref{eq:HardyEigenvalue}, we can write
$$\mathcal{B}(u_f,v) = \int_{\Omega}g(x,u_f)v\,dx\quad\forall\,\,v\in\mathcal{X}^{1,2}(\Omega),$$
where $g(x,t) = f(x)+\lambda\,t/|x|^2$; thus, since $f > 0$ a.e.\,$\Omega$, we have
$$\text{$g(x,t)\geq 0$ for a.e.\,$x\in\Omega$ and $t\geq 0$}.$$ 
Recalling \eqref{eq:u_f_geq_0}, we can apply
the Strong Maximum Principle arguing as in the proof of \cite[Corollary\,3.3]{BMV}, showing that $u_f > 0$ a.e.\,in $\Omega$
(as $f\not\equiv 0$).
\end{proof}
 Similarly to the case $\gamma = 0$ discussed in Section \ref{sec.Prel},
 see Theorem \ref{thm:generalDirichlet}\,-\,ii),
 also for the unique solution $u_f$  of problem \eqref{eq:HardyEigenvalue}
 we have an \emph{improved-integrability result}. 
\begin{thm}\label{thm:m**}
Let $f \in L^{m}(\Omega)$  with
$\tfrac{2n}{n+2}\leq m < \tfrac{n}{2},$ be a positive function.
If
$$0< \gamma < \gamma(m):= \dfrac{n(m-1)(n-2m)}{m^2},$$
\noindent then the unique weak solution $u_f\in\mathcal{X}^{1,2}(\Omega)$ of
problem \eqref{eq:HardyEigenvalue} 
\emph{(}whose existence is guaranteed by Theorem \ref{thm:LaxMilagrmVariational}\emph{)} is such that
$$u_f \in L^{m^{**}}(\Omega),\quad \textrm{where } m^{**}:= \tfrac{n m}{n-2m}$$
 \end{thm}
\begin{proof}
As it is customary for such kind of problems, for every $k\in\mathbb{N}$ we consider the unique weak 
solution $\varphi_k \in \mathcal{X}^{1,2}(\Omega) \cap L^{\infty}(\mathbb{R}^n)$ of the regularized problem
\begin{equation}\label{eq:TruncatedProblem}
\left\{\begin{array}{rl}
\mathcal{L}u = \gamma \dfrac{\varphi_{k-1}}{|x|^2 + 1/k} + f_k(x) & \textrm{in } \Omega,\\
u = 0 &\textrm{in } \mathbb{R}^{n}\setminus \Omega,
\end{array}\right.
\end{equation}
\noindent where $\varphi_{0}:=0$ and $f_k(x):= T_{k}(f(x))$. 
Owing to Lemma \ref{lem:convergenceApprox}, we know that
\begin{itemize}
\item[i)] $\varphi_k$ is well-defined for every $k\geq 1$, and $\{\varphi_k\}_k$ is non-decreasing;
\vspace*{0.1cm}

\item[ii)] $\varphi_k \to u_f$ as $k\to+\infty$ strongly in $L^{p}(\Omega)$ for every $p \in [1,2^{*})$, where 
$u_f\in\mathcal{X}^{1,2}(\Omega)$ is the unique weak solution of problem \eqref{eq:HardyEigenvalue}
(in the sense of Theorem \ref{thm:LaxMilagrmVariational});
\end{itemize}
 Let now 
\begin{equation}\label{eq:DefAlpha}
\alpha := \dfrac{m(n-2)}{n-2m},
\end{equation}
\noindent and notice the following:
\begin{itemize}
\item[a)] $m \geq \tfrac{2n}{n+2}$ is equivalent to $\alpha \geq 2$;
\item[b)] 
denoting by $m' = \tfrac{m}{m-1}$ the usual H\"{o}lder conjugate exponent of $m$, one has
$$\frac{2^{*}\alpha}{2} = \frac{n m}{n-2m} = m^{**}=(\alpha -1)m'.$$ 
\end{itemize}
Since $\varphi_{k}\in \mathcal{X}^{1,2}_{+}(\Omega)$ for every $k\in\mathbb{N}$, we can use $v_k := \varphi_{k}^{\alpha -1} 
(= |\varphi_k|^{\alpha -2}\varphi_k),$ as test function in the variational formulation of \eqref{eq:TruncatedProblem}, finding
\begin{equation}\label{eq:Testata}
\begin{aligned}
\mathcal{B}\left(\varphi_k, \varphi_k^{\alpha-1}\right) &= \gamma 
\int_{\Omega}\dfrac{\varphi_{k-1}\varphi_{k}^{\alpha-1}}{(|x|^2 + 1/k)} \, dx + 
\int_{\Omega}f_k \varphi_{k}^{\alpha-1}\, dx \\
&\leq \gamma \int_{\Omega}\dfrac{\varphi_{k}^{\alpha}}{|x|^2} \, dx+ \int_{\Omega}f_k \varphi_{k}^{\alpha-1}\, dx,
\end{aligned}
\end{equation}
\noindent where we used the increasing monotonicity of $\{\varphi_k\}_k$.

Let us now consider the left hand side of \eqref{eq:Testata}. We have
\begin{equation*}
\begin{aligned}
& \mathcal{B}\left(\varphi_k, \varphi_k^{\alpha-1}\right) \\
& \qquad = \int_{\Omega}\nabla \varphi_{k} \cdot \nabla (\varphi_{k}^{\alpha-1}) \, dx + 
\frac{C_{n,s}}{2} \iint_{\mathbb{R}^{2n}}
\dfrac{(\varphi_{k}(x)-\varphi_{k}(y))(\varphi_{k}^{\alpha-1}(x)-\varphi_{k}^{\alpha-1}(y))}{|x-y|^{n+2s}}\, dx\,dy\\
&\qquad =: (L) + (NL).
\end{aligned}
\end{equation*}
The nonlocal part can be treated as in the proof of \cite[Theorem 4.2]{AMPP}. Let us give the details for sake of completeness: first, exploiting the algebraic inequality \cite[Equation (17)]{AMPP}, i.e. for $s_1,s_2 \geq 0$ and $a>0$ it holds
$$(s_1-s_2)(s_1^a-s_2^a) \geq \frac{4a}{(a+1)^2} \left(s_1^\frac{a+1}{2} - s_2^\frac{a+1}{2}\right)^2,$$
we deduce that
\begin{equation*}
(\varphi_{k}(x)-\varphi_{k}(y))(\varphi_{k}^{\alpha-1}(x)-\varphi_{k}^{\alpha-1}(y)) \geq 
\dfrac{4(\alpha-1)}{\alpha^2}\left( \varphi_{k}^{\alpha/2}(x)-\varphi_{k}^{\alpha/2}(y)\right)^2,
\end{equation*}
\noindent and therefore
\begin{equation}\label{eq:NL}
(NL) \geq \dfrac{4(\alpha-1)}{\alpha^2} \left[ \varphi_{k}^{\alpha/2}\right]_s^2.
\end{equation}
In the local part we have,
\begin{equation*}
\nabla \varphi_{k} \cdot \nabla (\varphi_{k}^{\alpha-1}) = \dfrac{4(\alpha-1)}{\alpha^2} 
\left| \nabla \varphi_{k}^{\alpha/2}\right|^2,
\end{equation*}
\noindent and therefore
\begin{equation}\label{eq:L}
(L) = \dfrac{4(\alpha-1)}{\alpha^2} \int_{\Omega}\left| \nabla \varphi_{k}^{\alpha/2}\right|^2\, dx.
\end{equation}
\noindent Hence, combining \eqref{eq:L} and \eqref{eq:NL} we get
\begin{equation}\label{eq:Stima_B}
\mathcal{B}\left(\varphi_k, \varphi_k^{\alpha-1}\right) \geq 
\dfrac{4(\alpha-1)}{\alpha^2} \mathcal{B}\left(\varphi_{k}^{\alpha/2}, \varphi_{k}^{\alpha/2}\right) 
= \dfrac{4(\alpha-1)}{\alpha^2} \rho\left(\varphi_{k}^{\alpha/2}\right)^2.
\end{equation}
We now move to the right hand side of \eqref{eq:Testata}. By H\"{o}lder inequality, we find that
\begin{equation}\label{eq:Stima_f_k}
\int_{\Omega}f_k \varphi_{k}^{\alpha-1}\, dx \leq \|f_k\|_{L^{m}(\Omega)} \left( \int_{\Omega} \varphi_{k}^{(\alpha-1)m'}\, dx \right)^{1/m'},
\end{equation}
\noindent where $m'$ is the H\"{o}lder conjugate exponent of $m$ defined in b). It remains to deal with the last term: by the mixed Hardy inequality  with best constant, see Theorem \ref{thm:BestConstant}, we get
\begin{equation}\label{eq:Stima_Hardy_term}
\gamma \int_{\Omega}\dfrac{\varphi_{k}^{\alpha}}{|x|^2} \, dx  = 
\gamma \int_{\Omega} \dfrac{\left(\varphi_{k}^{\alpha/2}\right)^2}{|x|^2}\, dx 
\leq \dfrac{\gamma}{\Lambda_n} \rho\left(\varphi_{k}^{\alpha/2}\right)^2.
\end{equation}
Combining \eqref{eq:Testata} with \eqref{eq:Stima_B}, \eqref{eq:Stima_f_k} 
and \eqref{eq:Stima_Hardy_term}, we get
\begin{equation}\label{eq:PreSobolev}
\left(\dfrac{4(\alpha-1)}{\alpha^2} - \dfrac{\gamma}{\Lambda_n} \right) 
\rho\left(\varphi_{k}^{\alpha/2}\right)^2 \leq \|f_k\|_{L^{m}(\Omega)} \left( \int_{\Omega} 
\varphi_{k}^{(\alpha-1)m'}\, dx \right)^{1/m'}.
\end{equation}
As in the purely local case, in order to have a meaningful a priori estimate we need to have
\begin{equation*}
\gamma <\dfrac{4(\alpha-1)\Lambda_n}{\alpha^2} = \dfrac{n(m-1)(n-2m)}{m^2}= \gamma(m),
\end{equation*}
\noindent as assumed.
In order to conclude the proof we can follow the argument in \cite{BOP}. By \eqref{eq:Sobolevmista}, and recalling b), we get
\begin{equation*}
\rho\left(\varphi_{k}^{\alpha/2}\right)^2 \geq 
\mathcal{S}_{n}\|\varphi_{k}^{\alpha/2}\|_{L^{2^*}(\Omega)}^2 = 
\mathcal{S}_n \left(\int_{\Omega}\varphi_{k}^{(2^{*}\alpha) /2}\, dx\right)^{2/2^*}
\stackrel{b)}{=} \mathcal{S}_{n} \left( \int_{\Omega}\varphi_{k}^{(\alpha-1)m'}\,dx\right)^{2/2^*},
\end{equation*}
\noindent which, combined with \eqref{eq:PreSobolev} gives
\begin{equation*}
\begin{aligned}
\mathcal{S}_n & \left(\dfrac{4(\alpha-1)}{\alpha^2} - \dfrac{\gamma}{\Lambda_n} \right) 
\left( \int_{\Omega}\varphi_{k}^{(\alpha-1)m'}\,dx\right)^{2/2^* - 1/m'} \\
&= \mathcal{S}_n \left(\dfrac{4(\alpha-1)}{\alpha^2} - \dfrac{\gamma}{\Lambda_n} \right)\, 
\|\varphi_{k}\|_{L^{m^{**}}(\Omega)} \leq \|f_k\|_{L^{m}(\Omega)},
\end{aligned}
\end{equation*}
\noindent and thus
\begin{equation*}
\dfrac{\mathcal{S}_n}{\Lambda_n} \,\,(\gamma(m)-\gamma)\,\, 
 \|\varphi_{k}\|_{L^{m^{**}}(\Omega)} \leq \|f_k\|_{L^{m}(\Omega)}.
\end{equation*}
Now, by using the Fatou Lemma it is sufficient to pass to the limit for $k \rightarrow +\infty$ and note that 
$\varphi_k$ converges to the solution $u_f$ of problem \eqref{eq:HardyEigenvalue}. From this we deduce the thesis.
\end{proof}

We are now ready to prove Theorem \ref{thm:EnergySolutions}.

\begin{proof}[Proof of Theorem \ref{thm:EnergySolutions}]
It is enough to combine Theorem \ref{thm:LaxMilagrmVariational} with Theorem \ref{thm:m**}.
\end{proof}

\subsection{Case ii)\,-\,$f\in L^m(\Omega)$ with $1<m<(2^*)'$}
As already observed in Section \ref{sec.Prel},
in this case we cannot ensure that $f\in\mathbb{X} = (\mathcal{X}^{1,2}(\Omega))'$,
and thus we \emph{cannot apply} the Lax-Mil\-gram Theorem to study existence and uniqueness
of solutions of problem \eqref{eq:HardyEigenvalue}. Hence, we prove 
\emph{existence and improved integrability} at the same time with a truncation argument.

\begin{proof}[Proof of Theorem \ref{thm:W1m*}]
Similarly to the proof of Theorem \ref{thm:m**}, for every $k\in\mathbb{N}$ we consider 
the unique weak solution 
$\varphi_k\in \mathcal{X}^{1,2}_+(\Omega)\cap L^\infty(\R^n)$
of the regularized problem
\begin{equation}\label{eq:Ausiliario}
\left\{ \begin{array}{rl}
\mathcal{L}u = \gamma \dfrac{\varphi_{k-1}}{|x|^2 + 1/k} + f_k(x) & \textrm{in } \Omega,\\
u =0 & \textrm{in } \mathbb{R}^{n}\setminus \Omega,
\end{array}\right.
\end{equation}
\noindent where $\varphi_0\equiv 0$ and $f_{k}(x):= \min\{f(x),k\}$. 
By arguing exactly as in the \emph{incipit}
of the proof of Lemma \ref{lem:convergenceApprox}, we easily see that
 $\varphi_k$ is well-defined for every $k\geq 1$, and 
 the sequence $\{\varphi_k\}_k$ is non-decreasing.
 We then define the number $\alpha\in (1,2)$ as in \eqref{eq:DefAlpha},
 and we use
$$v_k := (\varphi_k + \varepsilon)^{\alpha-1} - \varepsilon^{\alpha-1}, \qquad \varepsilon>0,$$
\noindent as test function in the variational formulation of \eqref{eq:Ausiliario}. Recalling that
$$(\varphi_k +\varepsilon)(x)-(\varphi_k+\varepsilon)(y)= \varphi_{k}(x)-\varphi_{k}(y),$$
\noindent and making the same computations performed in the proof of Theorem \ref{thm:m**}, we get
\begin{equation}\label{eq:LHS}
\mathcal{B}\left(\varphi_k,v_k\right) = \mathcal{B}\left(\varphi_k,(\varphi_k+\varepsilon)^{\alpha-1}
- \varepsilon^{\alpha-1}\right)\geq \dfrac{4(\alpha-1)}{\alpha^2}
\rho\left((\varphi_k+\varepsilon)^{\alpha/2} - \varepsilon^{\alpha/2}\right)^2,
\end{equation}
\noindent and, using the monotonicity of $\{\varphi_k\}_k$ and the mixed Hardy inequality, we also find
\begin{equation}\label{eq:RHS}
\begin{aligned}
\gamma & \int_{\Omega}\dfrac{\varphi_k [(\varphi_k + \varepsilon)^{\alpha-1}-
\varepsilon^{\alpha-1}]}{|x|^2}\, dx + \int_{\Omega}f_k 
[(\varphi_k +\varepsilon)^{\alpha-1}-\varepsilon^{\alpha-1}]\, dx \\
&= \gamma \int_{\Omega}\dfrac{\left((\varphi_k +\varepsilon)^{\alpha/2} -
\varepsilon^{\alpha/2}\right)^2 }{|x|^2}\, dx \\
& \quad + \gamma \int_{\Omega}
\dfrac{\varphi_k [(\varphi_k + \varepsilon)^{\alpha-1}-\varepsilon^{\alpha-1}]-
[(\varphi_k +\varepsilon)^{\alpha/2} -\varepsilon^{\alpha/2}]^2}{|x|^2} \, dx\\
& \quad +\int_{\Omega}f_k 
 [(\varphi_k +\varepsilon)^{\alpha-1}-\varepsilon^{\alpha-1}]\, dx\\
&\leq \dfrac{\gamma}{\Lambda_n}
 \rho\left((\varphi_k+\varepsilon)^{\alpha/2}-\varepsilon^{\alpha/2}\right)^2 \\
& \quad +  \gamma \int_{\Omega}
 \dfrac{\varphi_k [(\varphi_k + \varepsilon)^{\alpha-1}-\varepsilon^{\alpha-1}]
  -[(\varphi_k +\varepsilon)^{\alpha/2} -\varepsilon^{\alpha/2}]^2}{|x|^2} \, dx \\
&\quad + \int_{\Omega}f_k [(\varphi_k +\varepsilon)^{\alpha-1}-\varepsilon^{\alpha-1}]\, dx\\
\end{aligned}
\end{equation}
Combining \eqref{eq:LHS}, \eqref{eq:RHS}, we finally get
\begin{equation}\label{eq:bdd1}
\begin{split}
\left( \dfrac{4(\alpha-1)}{\alpha^2} - \dfrac{\gamma}{\Lambda_n} \right) 
&\rho\left((\varphi_k+\varepsilon)^{\alpha/2} -\varepsilon^{\alpha/2}\right)^2 \leq   
\int_{\Omega}f_k [(\varphi_k +\varepsilon)^{\alpha-1}-\varepsilon^{\alpha-1}]\, dx\\
&+ \gamma \int_{\Omega}\dfrac{\varphi_k [(\varphi_k + \varepsilon)^{\alpha-1}-\varepsilon^{\alpha-1}]
 -[(\varphi_k +\varepsilon)^{\alpha/2} -\varepsilon^{\alpha/2}]^2}{|x|^2} \, dx.
\end{split}
\end{equation}
Hence, using \eqref{eq:Sobolevmista}, we get
\begin{equation*}
\begin{split}
	\rho\left((\varphi_k+\varepsilon)^{\alpha/2} -\varepsilon^{\alpha/2}\right)^2 & 
	\geq \mathcal{S}_{n}\|(\varphi_k+\varepsilon)^{\alpha/2} -\varepsilon^{\alpha/2}
	\|_{L^{2^*}(\Omega)}^2 \\
	& = \mathcal{S}_n \left(\int_{\Omega}
	[(\varphi_k+\varepsilon)^{\alpha/2} -\varepsilon^{\alpha/2}]^{2^*}\, dx\right)^{2/2^*}.
	\end{split}
\end{equation*}
Now, for $k \in \N$ fixed, thanks to the Lebesgue's dominated convergence theorem we can pass to the limit for $\varepsilon$ that goes to $0$, and thus we obtain
\begin{equation}\label{eq:lastepsilon}
	\mathcal{S}_{n} \left( \dfrac{4(\alpha-1)}{\alpha^2} - \dfrac{\gamma}{\Lambda_n} \right) 
	\left(\int_{\Omega}\varphi_{k}^{(2^{*}\alpha) /2}\, dx\right)^{2/2^*} 
	\leq \int_{\Omega}f_k \varphi_k^{\alpha-1}\, dx.
\end{equation}
Hence, applying H\"older inequality in the right hand side of \eqref{eq:lastepsilon}, we deduce
\begin{equation*}
	\mathcal{S}_{n} \left( \dfrac{4(\alpha-1)}{\alpha^2} - \dfrac{\gamma}{\Lambda_n} \right)
	 \left(\int_{\Omega}\varphi_{k}^{(2^{*}\alpha) /2}\, dx\right)^{2/2^*} 
	 \leq \|f_k\|_{L^m(\Omega)} \left(\int_{\Omega}  \varphi_k^{(\alpha-1)m'}\, dx \right)^\frac{1}{m'}.
\end{equation*}
Using the definition of $\alpha$ in \eqref{eq:DefAlpha}
and recalling that $\tfrac{\alpha}{2}\,2^{*} = \tfrac{n m}{n-2m} = m^{**}=(\alpha -1)m'$, we deduce
\begin{equation}\label{eq:bdd2}
	\mathcal{S}_{n} \left( \dfrac{4(\alpha-1)}{\alpha^2} - \dfrac{\gamma}{\Lambda_n} \right) 
	\|\varphi_k\|_{L^{m^{**}}(\Omega)}  \leq \|f_k\|_{L^m(\Omega)} \leq \|f\|_{L^m(\Omega)}.
\end{equation}
Hence, by \eqref{eq:bdd2} we deduce
\begin{equation}\label{eq:phik}
\|\varphi_k\|_{L^{m^{**}}(\Omega)}  \leq \mathcal{K}(n, \alpha, \gamma, f) =: \mathcal{K}.
\end{equation}
Exploiting the mixed Hardy inequality (see Theorem \ref{thm:Main2}) to the term $\rho\left((\varphi_k+\varepsilon)^{\alpha/2} -\varepsilon^{\alpha/2}\right)^2$, we have
\begin{equation}\label{eq:detail0}
\rho\left((\varphi_k+\varepsilon)^{\alpha/2} -\varepsilon^{\alpha/2}\right)^2 \geq   
\Lambda_n \int_{\Omega}\dfrac{[(\varphi_k +\varepsilon)^{\alpha/2} -\varepsilon^{\alpha/2}]^2}{|x|^2} \, dx.
\end{equation}
Going back to \eqref{eq:bdd1}, and using \eqref{eq:detail0}, we get

\begin{equation}\label{eq:detail1}
	\begin{split}
		\Lambda_n  &	\left( \dfrac{4(\alpha-1)}{\alpha^2} - \dfrac{\gamma}{\Lambda_n} \right)
		\int_{\Omega}\dfrac{[(\varphi_k +\varepsilon)^{\alpha/2} -\varepsilon^{\alpha/2}]^2}{|x|^2} \, dx \leq
		\left( \dfrac{4(\alpha-1)}{\alpha^2} - \dfrac{\gamma}{\Lambda_n} \right) 
		\rho\left((\varphi_k+\varepsilon)^{\alpha/2} -\varepsilon^{\alpha/2}\right)^2\\ \leq &
		\int_{\Omega}f_k [(\varphi_k +\varepsilon)^{\alpha-1}-\varepsilon^{\alpha-1}]\, dx + \gamma \int_{\Omega}\dfrac{\varphi_k [(\varphi_k + \varepsilon)^{\alpha-1}-\varepsilon^{\alpha-1}]
			-[(\varphi_k +\varepsilon)^{\alpha/2} -\varepsilon^{\alpha/2}]^2}{|x|^2} \, dx.
	\end{split}
\end{equation}
Let $k$ be fixed, by letting $\varepsilon \rightarrow 0$, we obtain
\begin{equation}\label{eq:detail2}
		\Lambda_n  	\left( \dfrac{4(\alpha-1)}{\alpha^2} - \dfrac{\gamma}{\Lambda_n} \right)
		\int_{\Omega}\dfrac{\varphi_k^{\alpha}}{|x|^2} \, dx \leq
		\int_{\Omega}f_k \cdot  \varphi_k^{\alpha-1}\, dx.
\end{equation}
Making use of H\"older inequality in the right hand side of \eqref{eq:detail2}, we deduce that
\begin{equation}\label{eq:detail3}
	\begin{split}
	\Lambda_n  	\left( \dfrac{4(\alpha-1)}{\alpha^2} - \dfrac{\gamma}{\Lambda_n} \right)
	\int_{\Omega}\dfrac{\varphi_k^{\alpha}}{|x|^2} \, dx &\leq
	\int_{\Omega}f_k \cdot  \varphi_k^{\alpha-1}\, dx \\
	&\leq \left(\int_{\Omega} f_k^m \,dx\right)^{1/m} \left(\int_{\Omega} \varphi_k^{(\alpha-1)m'}\, dx\right)^{1/m'}\\
	&\leq \|f\|_{L^m(\Omega)} \cdot \|\varphi_k\|_{L^{m^{**}(\Omega)}}^{m^{**}/m'}.
	\end{split}
\end{equation}
As a consequence of \eqref{eq:detail3}, taking into account \eqref{eq:phik}, we are able to conclude that
\begin{equation}\label{eq:uniformHardy}
	\int_\Omega \dfrac{\varphi_k^\alpha}{|x|^2} \,dx \leq \mathcal{E} \|f\|_{L^m(\Omega)},
\end{equation}
where $\mathcal{E}:=\mathcal{K}^{m^{**}/m'} \Lambda_n^{-1} (4(\alpha-1)/\alpha^2-\gamma/\Lambda_n)^{-1}$ is a positive constant which does not depend on $k$.

Our aim is to show that the sequence $\{\varphi_k\}_k$ is bounded in $W^{1,m^*}_0(\Omega)$. 
In order to prove this fact, from \eqref{eq:bdd1}, we have
\begin{equation}\label{eq:bdd3}
	\begin{split}
	&\left( \dfrac{4(\alpha-1)}{\alpha^2} - \dfrac{\gamma}{\Lambda_n} \right) 
	\left(\frac{\alpha}{2}\right)^2 \int_{\Omega} \frac{|\nabla \varphi_k|^2}
	{(\varphi_k+\varepsilon)^{2-\alpha}} \, dx \\
	& \qquad
	 \leq \left( \dfrac{4(\alpha-1)}{\alpha^2} - \dfrac{\gamma}{\Lambda_n} \right)
	 \rho\left((\varphi_k+\varepsilon)^{\alpha/2} - \varepsilon^{\alpha/2}\right)^2 \\
	&\qquad \leq  \gamma \int_{\Omega}
	\dfrac{\varphi_k [(\varphi_k + \varepsilon)^{\alpha-1}-\varepsilon^{\alpha-1}]-
	[(\varphi_k +\varepsilon)^{\alpha/2} -\varepsilon^{\alpha/2}]^2}{|x|^2} \, dx \\
	&\qquad\qquad+ \int_{\Omega}f_k [(\varphi_k +\varepsilon)^{\alpha-1}-\varepsilon^{\alpha-1}]\, dx\\
	& \qquad (\text{since $0<\alpha-1<1$})\\
	&\qquad \leq  \gamma \int_{\Omega}
	\dfrac{\varphi_k [\varphi_k^{\alpha-1} + \varepsilon^{\alpha-1}-\varepsilon^{\alpha-1}]}{|x|^2} \, dx + \int_{\Omega}f_k [\varphi_k^{\alpha-1} +\varepsilon^{\alpha-1}-\varepsilon^{\alpha-1}]\, dx\\
	&\qquad =  \gamma \int_{\Omega}
	\dfrac{\varphi_k^{\alpha}}{|x|^2} \, dx + \int_{\Omega}f_k \cdot \varphi_k^{\alpha-1}\, dx.\\ 
	\end{split}
\end{equation}
If we fix $\varepsilon=R$ in  \eqref{eq:bdd3}, using \eqref{eq:bdd2} and  \eqref{eq:uniformHardy}, we get
\begin{equation*}
	\begin{split}
		 \int_{\Omega} \frac{|\nabla \varphi_k|^2}
		 {(\varphi_k+R)^{2-\alpha}} \, dx \leq \mathcal{C}(m,n, \gamma, \mathcal{S}_n)=: \mathcal C.
	\end{split}
\end{equation*}
Thanks to H\"older inequality with exponents $(2/m^*,2/(2-m^*))$ we get
\begin{equation*}
	\begin{split}
	\int_{\Omega} |\nabla \varphi_k|^{m^*} \, dx & \leq 	
	\int_{\Omega} \frac{|\nabla \varphi_k|^{m^*}}
	{(\varphi_k+R)^\frac{(2-\alpha)m^*}{2}} (\varphi_k+R)^\frac{(2-\alpha)m^*}{2}\, dx \\
	& \leq \mathcal{C}^{\frac{m^*}{2}} \left(\int_{\Omega} 
	(\varphi_k+R)^\frac{(2-\alpha)m^*}{2-m^*}\, dx \right)^{\frac{2-m^*}{2}}.
	\end{split}
\end{equation*}
Finally, we have
\begin{equation*}
	\begin{split}
		\int_{\Omega} |\nabla \varphi_k|^{m^*} \, dx \leq \mathcal{C}^{\frac{m^*}{2}} 
		\left(\int_{\Omega} (\varphi_k+R)^{m^{**}}\, dx \right)^{\frac{2-m^*}{2}},
	\end{split}
\end{equation*}
which implies the boundedness of $\{\varphi_k\}_k$ in $W_0^{1,m^*}(\Omega)$. Since $m>1$,  we deduce that there exists $u$ such that (up to subsequences)
$$\varphi_k \rightharpoonup u \text{ in } W^{1,m^*}_0(\Omega)  \text{ and a.e. in } \Omega, 
\frac{\varphi_k}{|x|^2} \rightarrow \frac{u}{|x|^2} \text{ strongly in } L^1(\Omega).$$
We recall that $\varphi_k$ weakly solves \eqref{eq:Ausiliario}, namely
\[
\begin{split}
	\int_{\R^n} \nabla \varphi_k \cdot \nabla v \, dx &+ 
	\frac{C_{n,s}}{2} \iint_{\R^{2n}} \frac{(\varphi_k(x)-\varphi_k(y))(v(x)-v(y))}{|x-y|^{n+2s}} \, dx dy \\
	&= \gamma \int_\Omega \dfrac{\varphi_{k-1}}{|x|^2 + 1/k}v \, dx  + \int_\Omega f_k(x) v \, dx,
\end{split}
\]
for every $v \in C_0^\infty(\Omega)$.

We note that, for any $v \in C_0^\infty(\Omega)$ the following
$$W^{1,m^*}_0(\Omega) \ni g \mapsto \int_{\R^n} \nabla g \cdot \nabla v \, dx + \frac{C_{n,s}}{2} \iint_{\R^{2n}} \frac{(g(x)-g(y))(v(x)-v(y))}{|x-y|^{n+2s}} \, dx dy$$
is a bounded linear functional in $W^{1,m^*}_0(\Omega)$. Then for every $v \in C_0^\infty(\Omega)$, we obtain
\[
\begin{split}
&\int_{\R^n} \nabla \varphi_k \cdot \nabla v \, dx + 
\frac{C_{n,s}}{2} \iint_{\R^{2n}} \frac{(\varphi_k(x)-\varphi_k(y))(v(x)-v(y))}{|x-y|^{n+2s}} \, dx dy\\
& \qquad \longrightarrow \int_{\R^n} \nabla u \cdot \nabla v \, dx + \frac{C_{n,s}}{2} \iint_{\R^{2n}} \frac{(u(x)-u(y))(v(x)-v(y))}{|x-y|^{n+2s}} \, dx dy \text{ as } k \rightarrow +\infty.
\end{split}
\]
Collecting the last information on the operator, the convergence of $\varphi_k$ discussed before, and using the Lebesgue's dominated convergence theorem in the right hand side of the regularized problem \eqref{eq:Ausiliario}, we can pass to the limit in it, showing that $u$ is a solution (as defined in \eqref{eq:del_thm_1.5}) belonging to $W^{1,m^*}_0(\Omega)$. The solution $u$ is positive in $\Omega$ being $\{\varphi_k\}_k$  non-decreasing.

Now we want to prove that this solution $u \in W^{1,m^*}_0(\Omega)$, obtained as limit of solutions of the regularized problems \eqref{eq:Ausiliario}, is unique. Consider another sequence of functions $\{g_k\}_k$ converging to $f$ in $L^m(\Omega)$. Then let us consider the following regularized problem
\begin{equation*}
	\left\{ \begin{array}{rl}
		\mathcal{L}w_k = \gamma \dfrac{w_{k-1}}{|x|^2 + 1/k} + g_k(x) & \textrm{in } \Omega,\\
		w_k =0 & \textrm{in } \mathbb{R}^{n}\setminus \Omega.
	\end{array}\right.
\end{equation*}

Hence, as before we can prove that, up to subsequences, $w_k$ converges to a solution $w$ belonging to $W^{1,m^*}_0(\Omega)$. Hence, arguing on the function $\varphi_k-w_k$, we deduce that it solves
\begin{equation*}
	\left\{ \begin{array}{rl}
		\mathcal{L} (\varphi_k-w_k) = \gamma \dfrac{(\varphi_{k-1}-w_{k-1})}{|x|^2 + 1/k} + f_k(x)-g_k(x) & \textrm{in } \Omega,\\
		\varphi_k-w_k =0 & \textrm{in } \mathbb{R}^{n}\setminus \Omega.
	\end{array}\right.
\end{equation*}
Arguing as above, we can derive an analogous estimate to \eqref{eq:bdd2} for the difference function
\begin{equation*}
	\mathcal{S}_{n} \left( \dfrac{4(\alpha-1)}{\alpha^2} - \dfrac{\gamma}{\Lambda_n} \right) 
	\|\varphi_k-w_k\|_{L^{m^{**}}(\Omega)}  \leq \|f_k-g_k\|_{L^m(\Omega)}.
\end{equation*}
Passing to the limit for $k$ that goes to $+\infty$, we easily deduce that
\begin{equation*}
		 \|u-w\|_{L^{m^{**}}(\Omega)} =0,
\end{equation*}
which immediately implies the uniqueness.
\end{proof}

\subsection{Case iii)\,-\,$f\in L^1(\Omega)$}\label{subsec:caseiii}

Taking into account all the results established so far, we end this section
by proving an \emph{optimal condition} for the existence of a positive duality solution
of  problem \eqref{eq:HardyEigenvalue}  when $\beta\in (0,\Lambda_n)$ and $f$ is just a \emph{positive $L^1$ function} in the sense of Definition \ref{defin:duality}.

\begin{proof}[Proof of Theorem \ref{thm:solvabilityL1main}]
 (\emph{Necessity of \eqref{eq:conditionfPhi}}) We first assume that
 there exists a positive duality solution $u$ of problem
 \eqref{eq:HardyEigenvalue}, and we
 show that the
 `integra\-bility condition' \eqref{eq:conditionfPhi} is fulfilled by datum $f$.
   
 To this end (and similarly to the proof
 of Theorems \ref{thm:m**}\,-\,\ref{thm:W1m*}), for every $k\geq 1$ we consider the unique weak solution
$\phi_k\in \mathcal{X}_+^{1,2}(\Omega)\cap L^\infty(\R^n)$ of the problem
\begin{equation} \label{eq:PbApprox}
  \begin{cases}
   \LL u = g_k & \text{in $\Omega$}, \\
   u = 0 & \text{in $\R^n\setminus\Omega$},
  \end{cases}
 \end{equation}
 where $\phi_0\equiv 0$ and, to simplify the notation, we have set
 $$g_k(x) = \gamma\frac{\phi_{k-1}}{|x|^2+1/k}+1.$$
 Notice that, owing to Lemma \ref{lem:convergenceApprox} (with $f\equiv 1$), we have
 \begin{itemize}
  \item[a)] $\phi_k$ is well-defined for every $k\geq 1$, and the sequence
  $\{\phi_k\}_k$ is non-decreasing;
  \item[b)] $\phi_k\to \Phi_\Omega$ as $k\to+\infty$ in $L^p(\Omega)$, for every $1\leq p<2^*$.
 \end{itemize}
 Now, choosing  $w:=\phi_k$ and $g:=g_k$ in Definition \ref{defin:duality}
 (notice that this is legitimate, since we have $g_k\in L^\infty(\Omega)$ and
 $\phi_k\in \mathcal{X}_+^{1,2}(\Omega)\cap L^\infty(\R^n)$
 solves \eqref{eq:PbApprox}), we get
 \begin{equation} \label{eq:passareallimite}
  \begin{split}
   \int_\Omega f\phi_k\,dx & = \int_\Omega u \left(g_k-\gamma 
   \frac{\phi_k}{|x|^2}\right)  dx \\
   & (\text{by definition of $g_k$}) \\
   & = \int_\Omega  u\left(1+\gamma\frac{\phi_{k-1}}{|x|^2+1/k} - 
   \gamma\frac{\phi_k}{|x|^2}\right)dx \\
   & (\text{by monotonicity of $\phi_k$}) \\
   & \leq \int_\Omega u\left[1+\gamma \phi_k \left(\frac{1}{|x|^2+1/k} - 
   \frac{1}{|x|^2}\right)\right]dx \\
   & (\text{since $\phi_k$ is positive}) \\
   & \leq \int_\Omega u\,dx = \|u\|_{L^1(\Omega)} < \infty.
  \end{split}
 \end{equation}
  From this, by passing to the limit as $k\to\infty$ in \eqref{eq:passareallimite} with the aid of
 the Monotone Con\-ver\-gence Theorem (and recalling b)), we immediately obtain
 the desired  \eqref{eq:conditionfPhi}.
 \medskip
 
 \noindent \emph{Proof}. (\emph{Sufficiency of \eqref{eq:conditionfPhi}}).
 We now assume that $f$ satisfies condition \eqref{eq:conditionfPhi}, and we prove that
 there exists a solution $u$ of problem \eqref{eq:HardyEigenvalue}, further satisfying properties 1)\,-\,2).
  
 To this end, for every fixed $k\geq 1$ we consider once again
 the unique variational solution 
 $\varphi_k\in \mathcal{X}_+^{1,2}(\Omega)\cap L^\infty(\R^n)$ 
 of the approximated problem
 \begin{equation} \label{eq:PbApproxBis}
  \begin{cases}
   \LL u = \gamma\frac{\varphi_{k-1}}{|x|^2+1/k}+f_k & \text{in $\Omega$}, \\
   u = 0 & \text{in $\R^n\setminus\Omega$},
  \end{cases}
 \end{equation}
 where $\varphi_0\equiv 0$ and
 $f_k = T_k(f)$ (here, $T_k(\cdot)$ is the truncation operator
 defined in \eqref{eq:defTk}). 
 By arguing as in the \emph{incipit} of the proof
 of Lemma \ref{lem:convergenceApprox}, we see that
 $\varphi_k$ is well-defined for every $k\geq 1$, and the sequence
  $\{\varphi_k\}_k$ is non-decreasing; hence, we can define
 $$u(x) = \lim_{n\to\infty}\varphi_k(x) = \sup\{\varphi_k(x):\,n\in\mathbb{N}\}\geq 0.$$
 We then prove that this $u$ is a duality solution of \eqref{eq:HardyEigenvalue}, further satisfying 1)\,-\,2).
 
 \vspace*{0.1cm}
 
  a)\,\,\emph{$u$ is a duality solution of \eqref{eq:HardyEigenvalue}.} First of all, we show that $u\in L^1(\Omega)$.
  To this end it suffices to observe that, using $\varphi_k$ as a test
  function for the (variational) equation solved by $\Phi_\Omega$
  and viceversa, see \eqref{eq:weakform}, from the assumed condition
  \eqref{eq:conditionfPhi} we have 
 \begin{equation} \label{eq:stimauL1}
 \begin{split}
  \int_{\Omega}\varphi_k\,dx & = -\gamma\int_\Omega\frac{\Phi_\Omega \varphi_k}{|x|^2}\,dx+
  \mathcal{B}(\Phi_\Omega,\varphi_k) \\
  & (\text{since $\varphi_k$ solves \eqref{eq:PbApproxBis}}) \\
  & = \gamma\int_\Omega 
  \Phi_\Omega \Big(\frac{\varphi_{k-1}}{|x|^2+1/k}-\frac{\varphi_k}{|x|^2}\Big)dx
  +\int_\Omega f_k\Phi_\Omega\,dx \\
  & \leq \int_\Omega f_k\Phi_\Omega\,dx 
  \leq \int_\Omega f\Phi_\Omega\,dx = \mathbf{c}<\infty;
  \end{split}
 \end{equation}
 this, together with the Monotone Convergence theorem, proves at once that $u\in L^1(\Omega)$
 (in particu\-lar, $u < \infty$ a.e.\,in $\Omega$), and $\varphi_k\to u$ in $L^1(\Omega)$.
 
 Taking into account Definition \ref{defin:duality}, we now turn to prove that
 ${u}/{|x|^2}\in L^1(\Omega).$
 To this end, let $\phi_1\in \mathcal{X}^{1,2}_+(\Omega)\cap L^\infty(\Omega)$ be  the unique solution of 
 $$\begin{cases}
 \LL u = 1 & \text{in $\Omega$} \\
 u = 0 & \text{in $\R^n\setminus\Omega$}.
 \end{cases}$$
 Moreover, let
 $r > 0$ be such that $B_r(0)\Subset\Omega$.
 Owing to the Weak Harnack inequality in 
 \cite[Theorem 8.1]{GarainKinnunen}, we can find a constant $C > 0$, depending
 on $\phi_1$ and on $r$, such that
 $$\text{$\phi_1 \geq C$ a.e.\,in $B_r(0)$};$$
 as a consequence, using $\varphi_k$ as a test function for the equation solved by $\phi_1$
 and viceversa, from the above \eqref{eq:stimauL1} we obtain the following estimate
 \begin{equation} \label{eq:uoverxL1}
 \begin{split}
  \int_\Omega \frac{\varphi_{k-1}}{|x|^2+1/k}\,dx& \leq 
  \frac{1}{C}\int_{B_r(0)}\frac{\phi_1 \varphi_{k-1}}{|x|^2+1/k}\,dx
  + \frac{1}{r^2}\int_{\Omega\setminus B_r(0)}\varphi_{k-1}\,dx \\
  & (\text{since $\phi_1,\,\varphi_{k-1}\geq 0$ a.e.\,in $\Omega$}) \\
  & \leq \frac{1}{C}\int_{\Omega}\frac{\phi_1 \varphi_{k-1}}{|x|^2+1/k}\,dx
  + \frac{1}{r^2}\int_{\Omega}\varphi_{k-1}\,dx \\
  & (\text{using $\phi_1$ as a test function for the equation solved by $\varphi_k$}) \\
  & \leq \frac{1}{C} \left(\mathcal{B}(\varphi_{k},\phi_1)-\int_\Omega f_k \phi_1 \,dx \right)+
  \frac{1}{r^2}\int_{\Omega}\varphi_{k-1}\,dx \\
  & (\text{using $\varphi_k$ as a test function for the equation solved by $\phi_1$}) \\
  & \leq \frac{1}{C}\int_\Omega \varphi_k\,dx+\frac{1}{r^2}\int_\Omega \varphi_{k-1}\,dx \leq \mathbf{c} < \infty,
 \end{split}
 \end{equation}
 where $\mathbf{c} > 0$ is a constant independent of $k$.
 From this, again by the Monotone Convergence Theorem we infer that 
 \begin{equation*} 
  \frac{u}{|x|^2}\in L^1(\Omega)\quad\text{and}\quad
 \frac{\varphi_{k-1}}{|x|^2+1/k}\to \frac{u}{|x|^2}\,\,\text{in $L^1(\Omega)$}.
 \end{equation*}
 Gathering all these facts, we can easily prove that $u$ is a duality solution of
 problem \eqref{eq:HardyEigenvalue} in the sense of Definition \ref{defin:duality}: in fact, we already know that $u,\,u/|x|^2\in L^1(\Omega)$; 
 moreover, since $\varphi_k$ solves \eqref{eq:PbApproxBis}, we get
 $$\int_\Omega \mathcal{B}(\varphi_k,\psi) \,dx = \int_\Omega 
 \gamma\frac{\varphi_{k-1}\psi}{|x|^2+1/k}dx + \int_\Omega f_k\psi\,dx
 \quad\forall\,\,\psi\in \mathcal{X}^{1,2}(\Omega) \cap L^\infty(\Omega).$$
 If, in particular, $\psi$ is a solution of \eqref{eq:DirichletProblemaux}, we have
  $$\int_\Omega g \varphi_k = \int_\Omega \gamma\frac{\varphi_{k-1}\psi}{|x|^2+1/k}dx + 
  \int_\Omega f_k\psi\,dx.$$
 Then, letting $k\to\infty$ (and observing that $f_k \to f$ in $L^1(\Omega)$),
 by the above con\-siderations we conclude that $u$ is a solution of \eqref{eq:HardyEigenvalue}.
 \medskip
 
 b)\,\,\emph{$u$ satisfies \emph{1)\,-\,2)}}. We first prove that validity of property 1).
 To this end, we arbitra\-ri\-ly fix $j,k\in\mathbb{N}$ and we use
 $T_k(\varphi_j)$ as a test function for the equation solved by $\varphi_j$:
 taking into account the computation in \eqref{eq:uoverxL1}, we get
 \begin{align*}
  \int_\Omega |\nabla T_k(\varphi_j)|^2\,dx &
  \leq \mathcal{B}(\varphi_j,T_k(\varphi_j)) = \gamma\int_\Omega \frac{\varphi_{j-1}\,T_k(\varphi_j)}
  {|x|^2+1/j}\,dx
  +\int_\Omega f_j\,T_k(\varphi_j)\,dx \\
  & (\text{since $\varphi_{j-1},\,f_j\geq 0$ and $0\leq T_k(\varphi_k)\leq k$}) \\
  & \leq k\Big(\gamma\int_\Omega \frac{\varphi_{j-1}}{|x|^2+1/j}\,dx
  +\int_\Omega f\,dx\Big) \\
  & (\text{using \eqref{eq:uoverxL1}, and since $f\in L^1(\Omega)$}) \\
  & \leq \mathbf{c}_k < \infty,
 \end{align*}
 and this shows that the sequence $\{T_k(\varphi_j)\}_j$ is bounded in $\mathcal{X}^{1,2}(\Omega)$. Then,
 by the Sobolev Embedding Theorem there exists a function $g\in \mathcal{X}^{1,2}(\Omega)$ such that
 (up to a subsequence)
 \medskip
 
 i)\,\,$T_k(\varphi_j)\rightharpoonup g$ in $\mathcal{X}^{1,2}(\Omega)$ as $j\to\infty$;
 
 ii)\,\,$T_k(\varphi_j)\to g$ strongly in $L^2(\Omega)$ and pointwise a.e.\,in $\Omega$.
 \medskip
 
 \noindent On the other hand, since we also have $T_k(\varphi_j)\to T_k(u)$ pointiwse a.e.\,in $\Omega$
 as $j\to\infty$ (recall that, by definition, $u$ is the pointwise limit of $\varphi_j$), we conclude that 
 $$g = T_k(u),$$
 and this proves that $T_k(u)\in\mathcal{X}^{1,2}(\Omega)$ for every $k\in\mathbb{N}$, as desired.
 \vspace*{0.1cm}

 We then turn to prove the validity of property 2). To this end, we fix
 $\beta < 1/2$ (to be conveniently chosen later on), and we use the function
 $w_k=1-(1+\varphi_k)^{2\beta-1}$ as a test function for the equation solved
 by $\varphi_k$ (see, precisely, \eqref{eq:PbApproxBis}): this gives
 \begin{equation} \label{eq:doveRHSLHS}
  \mathcal{B}(\varphi_k,w_k) = \gamma\int_{\Omega}\frac{\varphi_{k-1}w_k}{|x|^2+1/k}\,dx
  +\int_\Omega f_kw_k\,dx.
 \end{equation}
 We now observe that, since
  $\beta < 1/2$, we have $w_k\leq 1$; as a consequence, for the second term of the right hand side
  of the above \eqref{eq:doveRHSLHS}, we have
 $$\int_\Omega f_k w_k \, dx \leq \int_\Omega f_k \, dx \leq \|f\|_{L^1(\Omega)}.$$ 
 For the second term, thanks to \eqref{eq:uoverxL1}, we have
 $$ \gamma \int_\Omega \frac{\varphi_{k-1}w_k}{|x|^2+1/k} \, dx \leq \gamma \int_\Omega 
 \frac{\varphi_{k-1}}{|x|^2+1/k}  \, dx \leq \mathbf{c}.$$
 Combining this last two information we deduce
 $$\|f\|_{L^1(\Omega)} + \mathbf{c} \geq \mathcal{B}(\varphi_k,w_k).$$
 Now, let us compute the bilinear form $\mathcal{B}(\varphi_k,w_k)$:
 \[
 \begin{split}
 	\mathcal{B}(\varphi_k,w_k)=  &|2\beta-1|\int_\Omega \frac{|\nabla \varphi_k|^2}{(1+\varphi_k)^{2(1-\beta)}} \,dx \\
 &+ \frac{C_{n,s}}{2}\iint_{\R^{2n}} \frac{(\varphi_k(x)-\varphi_k(y))((1+\varphi_k(y))^{2\beta-1} - (1+\varphi_k(x))^{2\beta-1})}{|x-y|^{n+2s}}.
 \end{split}
 \]
  We note that the nonlocal part of the bilinear form, i.e.\,the second term on the right hand side of the previous equation, is non-negative. This assertion follows from the following couple of facts
  $$\begin{cases}
  	\varphi_k(x)-\varphi_k(y) \geq 0 \ \Rightarrow (1+\varphi_k(y))^{2\beta-1} - (1+\varphi_k(x))^{2\beta-1} \geq 0,\\
  	\varphi_k(x)-\varphi_k(y) < 0 \ \Rightarrow (1+\varphi_k(y))^{2\beta-1} - (1+\varphi_k(x))^{2\beta-1} < 0.
  \end{cases}
  $$
  Hence, we are able to deduce that
  $$\|f\|_{L^1(\Omega)} + \mathbf{c} \geq 
  \mathcal{B}(\varphi_k,w_k) \geq  \frac{|2 \beta - 1|}{\beta^2} 
  \int_\Omega |\nabla (1+\varphi_k)^\beta|^2 \, dx.$$
 Thus from the previous estimate we get
 \begin{equation}\label{unifest}
 	\int_\Omega \frac{|\nabla \varphi_k|^2}{(1+\varphi_k)^{2(1-\beta)}}\, dx \leq \frac{\|f\|_{L^1(\Omega)} + \mathbf{c}}{|2\beta-1|}.
 \end{equation}
 On the other hand, for any given $p<2$ we can write
 $$ \int_\Omega |\nabla \varphi_k|^p \,dx = \int_\Omega \frac{|\nabla \varphi_k|^p}
 {(1+\varphi_k)^{2(1-\beta)\frac{p}{2}}} (1+\varphi_k)^{2(1-\beta)\frac{p}{2}}\, dx.$$
 Applying the H\"older inequality with exponents $(2/p,2/(2-p))$ on the right hand side, and the Sobolev inequality on the left hand side, we obtain
 \begin{align*}
  \mathcal{S}_n^p \left(\int_\Omega \varphi_k^{p^*} \, dx\right)^\frac{p}{p^*} & 
  \leq  \int_\Omega |\nabla \varphi_k|^p \,dx \\
  & \leq  \left(\int_\Omega \frac{|\nabla \varphi_k|^2}{(1+\varphi_k)^{2(1-\beta)}} \, dx\right)^{\frac{p}{2}} \left(\int_\Omega (1+\varphi_k)^{\frac{2p(1-\beta)}{2-p}} \, dx\right)^{1-\frac{p}{2}}.
  \end{align*}
  Thanks to \eqref{unifest}, we get
  \begin{equation}\label{gradest}
  	\mathcal{S}_n^p \left(\int_\Omega \varphi_k^{p^*} \, dx\right)^\frac{p}{p^*} \leq  
  	\int_\Omega |\nabla \varphi_k|^p \,dx \leq C |\Omega|^{1-\frac{p}{2}} + 
  	C\left(\int_\Omega \varphi_k^{\frac{2p(1-\beta)}{2-p}} \, dx\right)^{1-\frac{p}{2}},
  \end{equation}
  where $C= \left[(\|f\|_{L^1(\Omega)} + \mathbf{c})/|2\beta-1|\right]^{p/2}$, and it is independent on 
  $k$. 
  
  Now, we can choose $\beta$ such that
  $$\frac{2p(1-\beta)}{2-p}=p^*,$$
 and since $\beta <1/2$, we have $p< \frac{n}{n-1}$. Thanks to this choice of $\beta$, we have
 $$\mathcal{S}_n^p \left(\int_\Omega \varphi_k^{p^*} \, dx\right)^\frac{p}{p^*} \leq C 
 |\Omega|^{1-\frac{p}{2}} + C\left(\int_\Omega \varphi_k^{p^*} \, dx\right)^{1-\frac{p}{2}},$$
 thus we can deduce that $\varphi_k$ is uniformly bounded in $L^{p^*}(\Omega)$. From \eqref{gradest} we also deduce that $\nabla \varphi_k$ is uniformly bounded in $L^p(\Omega)$. Combining the last two information, we get that $\varphi_k$ is uniformly bounded in $W^{1,p}_0(\Omega)$. Arguing as before, it is possible to show that up to subsequences $\varphi_k \rightharpoonup u$ in $W^{1,p}_0(\Omega)$, for $p<\frac{n}{n-1}$. This closes the proof.
\end{proof}

\appendix
\section{A convergence lemma}

The aim of this appendix is to give an auxiliary result which clarify the convergence of the sequence $\{u_k\}$ in in the proof of Theorem \ref{thm:W1m*}.

\begin{lem}[Approximated Dirichlet problems] \label{lem:convergenceApprox}
Let $\Omega\subseteq\R^n$ be a bounded open set with smooth enough boundary $\de\Omega$,
and let 
$f\in L^{\frac{2n}{n+2}}(\Omega),\,\text{$f \geq 0$ a.e.\,in $\Omega$}.$
For a fixed $\gamma\in (0,\Lambda_n)$, we let $u\in\mathcal{X}^{1,2}_+(\Omega)$ be the unique weak solution of
\begin{equation} \label{eq:PbDir}
	\begin{cases}
		\LL u = \gamma\frac{u}{|x|^2}+f & \text{in $\Omega$}, \\
		u = 0 & \text{in $\R^n\setminus\Omega$}.
	\end{cases}
\end{equation}
Setting $\varphi_0 = 0$, for every $k\geq 1$ we then consider
the unique weak solution $\varphi_k\in\mathcal{X}^{1,2}_+(\Omega) \cap L^\infty(\R^n)$ of the
following \emph{approximated} $\LL$\,-\,Dirichlet problem
\begin{equation} \label{eq:PbDirApprox}
	\begin{cases}
		\LL u = \gamma\frac{\varphi_{k-1}}{|x|^2+1/k}+f_k & \text{in $\Omega$}, \\
		u = 0 & \text{in $\R^n\setminus\Omega$},
	\end{cases}
\end{equation}
where $f_k = T_k(f) = \min\{k,f\}$. Then, 
$$\text{$\varphi_k\to u$ weakly in $\mathcal{X}^{1,2}(\Omega)$ and strongly in $L^p(\Omega)$
	for every $1\leq p<2^*$}.$$
	\end{lem}
	\begin{proof}
    We begin with a couple of
    preliminary observations. First of all, the existence and the uniqueness of the solution $\varphi_k
 \in  \mathcal{X}_+^{1,2}(\Omega)\cap L^\infty(\R^n)$ follow
 from Theorem \ref{thm:generalDirichlet} (and from the 
 recursive structure of problem \eqref{eq:PbDirApprox}), since the right-hand side
 $$g_k(x) = \gamma\frac{\varphi_{k-1}}{|x|^2+1/k}+f_k$$
 is non-negative and bounded on $\Omega$ (as, by induction, the same is true of $\varphi_{k-1}$).	
 Furthermore, by u\-sing a classical induction argument based on the Weak Maximum Principle
for $\LL$ (and since the sequence $\{f_k\}_k$ is non-decreasing), we have that
$$\text{$\{\varphi_k(x)\}_k$ is non-negative and non-decreasing for every $x\in \Omega$}.$$  
As a consequence, we can define 
$$\text{$\psi(x) = \lim_{k\to\infty}\varphi_k(x)\in(0,+\infty]$ for all $x\in\Omega$.}$$
Now, using $\varphi_k$ as a test function for the equation solved by $u$
and viceversa, we have
\begin{align*}
	\int_{\Omega}f_k\varphi_k\,dx & \leq \int_{\Omega}f\varphi_k\,dx 
	= \mathcal{B}(u,\varphi_k)-\gamma\int_\Omega\frac{u\varphi_k}{|x|^2}\,dx \\
	& (\text{since $\varphi_k\geq\varphi_{k-1}$ and $|x|^2\leq |x|^2+1/k$}) \\
	& \leq \mathcal{B}(\varphi_k,u)-\gamma\int_\Omega\frac{\varphi_{k-1}u}{|x|^2+1/k}\,dx \\
	& = \int_\Omega f_ku\,dx \leq \int_\Omega fu\,dx;
\end{align*}
hence, by the Monotone Convergence Theorem we derive that
$$\text{$f\psi\in L^1(\Omega)$}.$$
On account of this fact, and using $\varphi_k$ as a test function
for the equation solved by $\varphi_k$ itself, we obtain the following estimate
\begin{align*}
	\rho(\varphi_k)^2 & = \mathcal{B}(\varphi_k,\varphi_k)
	= \gamma\int_\Omega\frac{\varphi_k\varphi_{k-1}}{|x|^2+1/k}\,dx
	+ \int_{\Omega}f_k\varphi_k\,dx \\
	& (\text{since $\varphi_{k-1}\leq\varphi_k\leq\psi$ a.e.\,in $\Omega$}) \\
	& \leq
	\gamma\int_\Omega\frac{\varphi_k^2}{|x|^2}\,dx
	+ \int_{\Omega}f\psi\,dx \\
	& (\text{by Hardy's inequality}) \\
	& \leq \frac{\gamma}{\Lambda_n} 
	\rho(\varphi_k)^2 + \|f\psi\|_{L^1(\Omega)}.
\end{align*}
This, together with the fact that $\gamma\in (0,\Lambda_n)$, shows that 
the sequence $\{\varphi_k\}_k$ is \emph{bounded in the Hilbert
	space $\mathcal{X}^{1,2}(\Omega)$}; hence, recalling that $\varphi_k\to\psi$ pointwise in $\Omega$
and using the Sobolev Embedding Theorem, we deduce that $\psi\in\mathcal{X}^{1,2}(\Omega)$ and
$$\text{$\varphi_k\to \psi$ weakly in $\mathcal{X}^{1,2}(\Omega)$ and strongly in $L^p(\Omega)$
	for every $1\leq p<2^*$}.$$
To complete the proof it now suffices to show that $\psi\in\mathcal{X}^{1,2}(\Omega)$ solves
the same equation solved by $u$; by uniqueness, this will prove that $\psi = u$, as desired.
Let then $\varphi\in\mathcal{X}^{1,2}(\Omega)$ be arbitrarily fixed. Since $\varphi_k$
is a weak solution of \eqref{eq:PbDirApprox}, we can write
\begin{equation} \label{eq:topassLimit}
	\mathcal{B}(\varphi_k,\varphi) = \gamma\int_\Omega\frac{\varphi_{k-1}\varphi}{|x|^2+1/k}\,dx
	+\int_\Omega f_k\varphi\,dx.
\end{equation}
Since $\varphi_k\rightharpoonup \psi$ in $\mathcal{X}^{1,2}(\Omega)$, we have
$$\mathcal{B}(\varphi_k,\varphi)\to\mathcal{B}(\psi,\varphi)\quad\text{as $k\to\infty$}.$$
Moreover, since $0\leq\varphi_k\leq \psi$, we have
\begin{align*}
	\Big|\frac{\varphi_{k-1}\varphi}{|x|^2+1/k}\Big|&\leq 
	\frac{\psi|\varphi|}{|x|^2} = g(x)\quad\forall\,\,k\geq 1,
\end{align*}
and $g\in L^1(\Omega)$ by H\"older's and Hardy's inequality 
(recall that $\psi,\varphi\in\mathcal{X}^{1,2}(\Omega)$). We are then entitled to
apply the Lebesgue Dominated Convergence theorem,
showing that
$$\int_\Omega\frac{\varphi_{k-1}\varphi}{|x|^2+1/k}\,dx\to 
\int_\Omega\frac{\psi\varphi}{|x|^2}\,dx.$$
Finally, since $0\leq f_k\leq f$, we also have
\begin{align*}
	|f_k\varphi|&\leq 
	f\varphi = h(x)\quad\forall\,\,k\geq 1,
\end{align*}
and $h\in L^1(\Omega)$ by H\"older's inequality 
(recall that $f\in L^{\frac{2n}{n+2}}(\Omega)$ and $\varphi\in\mathcal{X}^{1,2}(\Omega)
\subset L^{2^*}(\Omega)$). Again by the Lebesgue Dominated Convergence theorem,
we then get
$$\int_\Omega f_k\varphi\,dx\to\int_\Omega f\varphi\,dx.$$
Gathering all these facts, we can finally pass to the limit as $k\to\infty$ in the above
\eqref{eq:topassLimit},
thus deducing that $\psi$ is a weak solution of the same equation solved by $u$.
\end{proof}


\begin{thebibliography}{99}

\bibitem{AMPP}
B. Abdellaoui, M. Medina, I. Peral, A. Primo, 
{\em The effect of the Hardy potential in some Calder\'{o}n-Zygmund properties for the fractional Laplacian},
J. Differential Equations {\bf 260}(11), (2016), 8160--8206.

\bibitem{AP} 
B. Abdellaoui, I. Peral,
{\em A note on a critical problem with natural growth in the gradient},
J. Eur. Math. Soc. (JEMS) {\bf 8}(2), 2006, 157--170.

\bibitem{AntoCozzi}
C.A. Antonini, M. Cozzi,
{\em Global gradient regularity and a Hopf lemma for quasilinear operators of 
mixed local-nonlocal type},
preprint. \url{https://arxiv.org/abs/2308.06075}




\bibitem{ArRa}
R. Arora, V. Radulescu,
\emph{Combined effects in mixed local-nonlocal stationary problems}, Proceedings of the Royal Society of Edinburgh: Section A Mathematics. Published online 2023:1-47. doi:10.1017/prm.2023.801.06701

\bibitem{BDVV}
S. Biagi, S. Dipierro, E. Valdinoci, E. Vecchi,
{\em Mixed local and nonlocal elliptic operators: regularity and maximum principles}, 
Comm. Partial Differential Equations {\bf 47}(3) (2022), 585--629.

\bibitem{BDVV3}
S. Biagi, S. Dipierro, E. Valdinoci, E. Vecchi,
{\em A Faber-Krahn inequality for mixed local and nonlocal operators},
J. Anal. Math. {\bf 150}(2), (2023), 405--448.

\bibitem{BDVV5}
S.~Biagi, S.~Dipierro, E.~Valdinoci, E.~Vecchi, 
{\em A Brezis-Nirenberg type result for mixed local and nonlocal operators}, preprint. \url{https://arxiv.org/abs/2209.07502}

\bibitem{BMV}
S.~Biagi, D. Mugnai, E.~Vecchi, 
{\em A Brezis–Oswald approach for mixed local
and nonlocal operators}, 
Commun. Contemp. Math. {\bf 26}(2), (2024), 2250057.

\bibitem{BV2}
S. Biagi, E. Vecchi,
{\em Multiplicity of positive solutions for mixed local-nonlocal singular critical problems},
preprint. \url{https://arxiv.org/abs/2308.09794}

\bibitem{BOP}
L. Boccardo, L. Orsina, I. Peral, 
{\em A remark on existence and optimal summability of solutions of elliptic problems involving Hardy potential},
Discrete Contin. Dyn. Syst. {\bf 16}(3), (2006), 513--523.

\bibitem{BCP}
J.-M. Bony, P. Courrège, P. Priouret, 
{\em Semi-groupes de Feller sur une variété à bord compacte et problèmes aux limites intégro-différentiels du second ordre donnant lieu au principe du maximum},
Ann. Inst. Fourier (Grenoble) {\bf 18}(2), (1968), 369--521.

\bibitem{Cancelier}
C. Cancelier, 
{\em Problèmes aux limites pseudo-différentiels donnant lieu au principe du maximum},
Comm. Partial Differential Equations {\bf 11}(15), (1986), 1677--1726.

\bibitem{CKSV}
 Z.-Q. Chen, P. Kim, R. Song, Z. Vondra\v{c}ek, 
 {\em Boundary Harnack principle for $\Delta + \Delta^{\alpha/2}$}, 
 Trans. Amer. Math. Soc. {\bf 364}(8), (2012), 4169--4205. 

 \bibitem{DeFMin}
 C. De Filippis, G. Mingione, 
 {\em Gradient regularity in mixed local and nonlocal problems}, Math. Ann.{\bf 388}, (2024), 261--328. 

\bibitem{DRV}
E. Di Nezza, G. Palatucci, E. Valdinoci,
{\em Hitchhiker's guide to the fractional Sobolev spaces}, 
Bull. Sci. Math. {\bf 136} (2012), 521--573.

\bibitem{DPLV2}
S. Dipierro, E. Proietti Lippi, E. Valdinoci, 
{\em (Non)local logistic equations with Neumann conditions}, Ann. Inst. H. Poincaré Anal. Non Linéaire (2022), 
DOI: 10.4171/AIHPC/57 .

\bibitem{FallMus}
M. M. Fall, R. Musina,
{\em Sharp Nonexistence Results for
a Linear Elliptic Inequality Involving
Hardy and Leray Potentials}, 
Journal of Inequalities and Applications {\bf 2011}.

\bibitem{Garain}
P. Garain,
{\em On a class of mixed local and nonlocal semilinear elliptic equation with singular nonlinearity}, J. Geom. Anal. {\bf 33}, 212, (2023).

\bibitem{GarainKinnunen}
P. Garain, J. Kinnunen,
{\em On the regularity theory for mixed local and nonlocal quasilinear elliptic equations}, Trans. Amer. Math. Soc. {\bf 375}(8), (2022),5393--5423.

\bibitem{GarainLindgren}
P. Garain, E. Lindgren, 
{\em Higher Hölder regularity for mixed local and nonlocal degenerate elliptic equations}, Calc. Var. {\bf 62}, 67, (2023).

\bibitem{Hardy_Littlewood}
G. H. Hardy, J. E. Littlewood, G. Polya, 
Inequalities. Reprint of the 1952 edition. 
Cambridge
Mathematical Library. Cambridge University Press, Cambridge, 1988.




\bibitem{LPPS}
T. Leonori, I. Peral, A. Primo, F. Soria, 
{\em Basic estimates for solutions of a class of nonlocal elliptic and parabolic equations},
Discrete Contin. Dyn. Syst. {\bf 35}(12), (2015), 6031--6068.

\bibitem{PeralBook}
I. Peral, F. Soria,
{\em Elliptic and Parabolic Equations Involving the Hardy-Leray Potential},
De Gruyter Series in Nonlinear Analysis and Applications, Volume 38, (2021) Walter de Gruyter GmbH, Berlin/Boston.


\bibitem{Stampacchia}
G. Stampacchia, 
{\em Le probl\`eme de Dirichlet pour les \'{e}quations elliptiques du second ordre \`a coefficients discontinus}, 
Ann. Inst. Fourier (Grenoble) {\bf 15}, (1965), 189--258.

\bibitem{SVWZ} 
X. Su, E. Valdinoci, Y. Wei, J. Zhang,
{\em Regularity results for solutions of mixed local and nonlocal elliptic equations}, Math. Z. {\bf 302}, (2022), 1855--1878.

\bibitem{SVWZ2} 
X. Su, E. Valdinoci, Y. Wei, J. Zhang,
{\em On Some Regularity Properties of Mixed Local and Nonlocal Elliptic Equations}, preprint. \url{https://papers.ssrn.com/sol3/papers.cfm?abstract_id=4617397}

\bibitem{VazZua} 
J.L. Vazquez, E. Zuazua,
{\em The Hardy inequality and the asymptotic behaviour of the heat equation with an inverse-square potential}, J. Funct. Anal. {\bf 173}, (2000), 103--153.


\end{thebibliography}
\end{document}